\documentclass[letterpaper,11pt]{amsart}

\usepackage{amsmath}
\usepackage{amssymb}
\usepackage{amsfonts}
\usepackage{amsthm}
\usepackage{amsopn}
\usepackage{enumerate}
\usepackage[all]{xy}
\usepackage{array}
\usepackage{mathrsfs}
\usepackage{tikz}
\usepackage{dsfont}
\usepackage{caption}
\usepackage{fullpage}
\usepackage{hyperref}
\hypersetup{
 colorlinks  = true,  
 urlcolor   = blue,  
 linkcolor  = blue,  
 citecolor  = blue   
}

\mathchardef\mhyphen="2D
\newcommand{\m}{\alpha}

\newcommand{\kk}{\mathds{k}}
\newcommand{\NN}{\ensuremath{\mathbb{N}}}
\newcommand{\NNp}{\ensuremath{\mathbb{N}^+}}

\newcommand{\ZZ}{\ensuremath{\mathbb{Z}}}

\newcommand{\FFF}{\ensuremath{\mathcal{F}}}
\newcommand{\GGG}{\ensuremath{\mathcal{G}}}

\newcommand{\SSS}{\ensuremath{\mathcal{S}}}
\newcommand{\III}{\ensuremath{\mathcal{I}}}
\newcommand{\PPPP}{\ensuremath{\mathcal{P}}}

\newcommand{\CCC}{\ensuremath{\mathsf{C}}}
\newcommand{\DDD}{\ensuremath{\mathsf{D}}}
\newcommand{\PPP}{\ensuremath{\mathsf{P}}}
\newcommand{\RRR}{\ensuremath{\mathsf{R}}}

\newcommand{\ra}{\ensuremath{\longrightarrow}}
\newcommand{\sra}{\ensuremath{\rightarrow}}

\newcommand{\Ext}{\operatorname{Ext}}
\newcommand{\uExt}{\operatorname{\underline{Ext}}}
\newcommand{\Hom}{\operatorname{Hom}}
\newcommand{\uHom}{\operatorname{\underline{Hom}}}

\newcommand{\pd}{\operatorname{pd}}
\newcommand{\gldim}{\operatorname{gl dim}}

\newcommand{\ol}{\overline}

\newcommand{\Id}{\operatorname{Id}}
\newcommand{\im}{\operatorname{im}}
\newcommand{\coker}{\operatorname{coker}}

\def\lcm{\mathop{\operatorname{lcm}}}

\newcommand{\gr}{\operatorname{gr}\mhyphen}

\newcommand{\qgr}{\operatorname{qgr}\mhyphen}

\newcommand{\QgrA}{Q_{\mathrm{gr}}(A)}
\newcommand{\QgrAf}{Q_{\mathrm{gr}}\left(A(f)\right)}
\newcommand{\tor}{\operatorname{tor}}
\renewcommand{\mod}{\mathrm{mod\mhyphen}}

\newcommand{\fin}{\mathrm{fin}}
\newcommand{\Aut}{\operatorname{Aut}}
\newcommand{\Pic}{\operatorname{Pic}}
\newcommand{\s}[1]{\ensuremath{\langle #1 \rangle}}
\DeclareMathOperator{\Spec}{Spec}

\DeclareMathOperator{\Supp}{Supp}

\numberwithin{equation}{section}

\newtheorem{theorem}[equation]{Theorem}
\newtheorem{corollary}[equation]{Corollary}
\newtheorem{lemma}[equation]{Lemma}
\newtheorem{proposition}[equation]{Proposition}

\theoremstyle{remark}
\newtheorem{remark}[equation]{Remark}

\theoremstyle{definition}
\newtheorem{definition}[equation]{Definition}

\newtheorem{example}[equation]{Example}

\newtheorem*{acknowledgments}{Acknowledgments}

\begin{document}
\title{The Picard group of the graded module category of a generalized Weyl algebra}
\author{Robert Won}
\address{Department of Mathematics, UCSD, La Jolla, CA 92093-0112, USA.}
\email{rwon@math.ucsd.edu}

\thanks{The author was partially supported by NSF grant DMS-1201572.}
\date{}

\maketitle

\begin{abstract}
The first Weyl algebra, $A_1 = \kk \langle x, y\rangle/(xy-yx - 1)$ is naturally $\ZZ$-graded by letting $\deg x = 1$ and $\deg y = -1$. Sue Sierra studied $\gr A_1$, category of graded right $A_1$-modules, computing its Picard group and classifying all rings graded equivalent to $A_1$. In this paper, we generalize these results by studying the graded module category of certain generalized Weyl algebras. We show that for a generalized Weyl algebra $A(f)$ with base ring $\kk[z]$ defined by a quadratic polynomial $f$, the Picard group of $\gr A(f)$ is isomorphic to the Picard group of $\gr A_1$. In a companion paper, we use these results to construct commutative rings which are graded equivalent to generalized Weyl algebras.
\end{abstract}

\section{Introduction}

Throughout this paper, fix an algebraically closed field $\kk$ of characteristic zero. All vector spaces and algebras are taken over $\kk$ and all categories and equivalences of categories are $\kk$-linear.

The first Weyl algebra, $A_1 = \kk \langle x, y\rangle/(xy-yx - 1)$ is a fundamental example in non-commutative ring theory. There is a natural $\ZZ$-grading on $A_1$ given by letting $\deg x = 1$ and $\deg y = -1$. In \cite{sierra}, Sierra studied $\gr A_1$, the category of finitely generated graded right $A_1$-modules. She determined the group of autoequivalences of $\gr A_1$ modulo natural isomorphism (called the Picard group of $\gr A_1$) and classified all rings graded equivalent to $A_1$. The simple modules of $\gr A_1$ can be pictured as follows. For each $\lambda \in \kk \setminus \ZZ$, there is a simple module $M_\lambda = A_1/(xy+\lambda)A_1$, while for each $n \in \ZZ$, the module $A_1/(xy+n)A_1$ is the nonsplit extension of two simple modules $X\s{n}$ and $Y\s{n}$. We can therefore represent the simple modules as the affine line with integer points doubled. Sierra gives the following picture:

\centerline{
	\begin{tikzpicture}[ scale=1.6]	
			\node at (-1.7,.1)[label = above: {}]{};
		\draw[thick,<-](-3.9,0)--(-3.1,0);
		\fill[black](-3,-.1) circle (1pt);
		\fill[black](-3,.1) circle (1pt);
		\node at (-3,-.1)[label = below: $-3$]{};
		\draw[thick](-2.9,0)--(-2.1,0);
		\fill[black](-2,-.1) circle (1pt);
		\fill[black](-2,.1) circle (1pt);
		\node at (-2,-.1)[label = below: $-2$]{};
		\draw[thick](-1.9,0)--(-1.1,0);
		\fill[black](-1,-.1) circle (1pt);
		\fill[black](-1,.1) circle (1pt);
		\node at (-1,-.1)[label = below:$-1$]{};
		\draw[thick](-.9,0)--(-.1,0);
		\fill[black](0,-.1) circle (1pt);
		\fill[black](0,.1) circle (1pt);
		\node at (0,-.1)[label = below: $0$]{};
		\draw[thick](.1,0)--(.9,0);
		\fill[black](1,-.1) circle (1pt);
		\fill[black](1,.1) circle (1pt);
		\node at (1,-.1)[label = below: $1$]{};
		\draw[thick](1.1,0)--(1.9,0);		
		\fill[black](2,-.1) circle (1pt);
		\fill[black](2,.1) circle (1pt);
		\node at (2,-.1)[label = below: $2$]{};
		\draw[thick,](2.1,0)--(2.9,0);		
		\fill[black](3,-.1) circle (1pt);
		\fill[black](3,.1) circle (1pt);
		\node at (3,-.1)[label = below: $3$]{};
		\draw[thick,->](3.1,0)--(3.9,0);		
	\end{tikzpicture}
	}

In determining the Picard group of $\gr A_1$, Sierra showed that there are many symmetries of this picture. There is an autoequivalence of $\gr A_1$ induced by a graded automorphism of $A_1$ which reflects the picture, sending $A_1/(xy+n)A_1$ to $A_1/(xy - (n+1))A_1$ for all $n \in \ZZ$. The shift functor $\SSS_{A_1}$ is an autoequivalence of the category which translates the picture, sending $A_1/(xy+n)A_1$ to $A_1/(xy + n+1)A_1$ for all $n \in \ZZ$. Hence, $\Pic(\gr A_1)$ has a subgroup isomorphic to $D_\infty$. Further, Sierra constructed autoequivalences $\iota_n$ of $\gr A_1$, which permute $X\s{n}$ and $Y\s{n}$ and fix all other simple modules. Let $\ZZ_\fin$ denote the group of finite subsets of the integers, with operation exclusive or. The subgroup of $\Pic(\gr A_1)$ generated by the $\iota_n$ is isomorphic to $\ZZ_\fin$. 

\begin{theorem}[Sierra {\cite[Corollary 5.11]{sierra}}] The group $\Pic(\gr A_1)$ is isomorphic to $\ZZ_\fin \rtimes D_{\infty}$.
\end{theorem}

In this paper, we study the graded module category over certain generalized Weyl algebras (GWAs). We study GWAs $A(f)$ with base ring $\kk[z]$ defined a quadratic polynomial $f \in \kk[z]$ and automorphism $\sigma: \kk[z] \sra \kk[z]$ mapping $z$ to $z+1$. The GWA $A(f)$ has presentation 
\[A(f) = \frac{\kk[z] \langle x,y \rangle}{\begin{pmatrix}xz = \sigma(z)x & yz = \sigma^{-1}(z)y \\ xy = f & yx = \sigma^{-1}(f)\end{pmatrix}}.
\]
We assume, without loss of generality, that $f = z(z+\m )$ for some $\m \in \kk$. The properties of $A(f)$ are determined by the roots of $f$. When $\m = 0$, we say that $f$ has a multiple root, when $\m \in \NNp$, we say that $f$ has congruent roots, and when $\m \in \kk \setminus \ZZ$, we say that $f$ has non-congruent roots.
In the non-congruent root case, the picture of $\gr A(f)$ can be thought of as a ``doubled'' version of the picture of $\gr A_1$.

\centerline{
	\begin{tikzpicture}[ scale=1.6]	
		\draw[thick,<-](-3.9,0)--(-3.1,0);
		\fill[black](-3,-.1) circle (1pt);
		\fill[black](-3,.1) circle (1pt);
		\node at (-3,-.1)[label = below: $-3$]{};
		\draw[thick](-2.9,0)--(-2.8,0);
				\fill[black](-2.7,-.1) circle (1pt);
		\fill[black](-2.7,.1) circle (1pt);
		\node at (-2.7,.1)[label = above: $\m -3$]{};
		\draw[thick](-2.6,0)--(-2.1,0);
		\fill[black](-2,-.1) circle (1pt);
		\fill[black](-2,.1) circle (1pt);
		\node at (-2,-.1)[label = below: $-2$]{};
		\draw[thick](-1.9,0)--(-1.8,0);
				\fill[black](-1.7,-.1) circle (1pt);
		\fill[black](-1.7,.1) circle (1pt);
		\node at (-1.7,.1)[label = above: $\m -2$]{};
		\draw[thick](-1.6,0)--(-1.1,0);
		\fill[black](-1,-.1) circle (1pt);
		\fill[black](-1,.1) circle (1pt);
		\node at (-1,-.1)[label = below:$-1$]{};
		\draw[thick](-.9,0)--(-.8,0);
			\fill[black](-.7,-.1) circle (1pt);
		\fill[black](-.7,.1) circle (1pt);
		\node at (-.7,.1)[label = above:$\m -1$]{};
		\draw[thick](-.6,0)--(-.1,0);
		\fill[black](0,-.1) circle (1pt);
		\fill[black](0,.1) circle (1pt);
		\node at (0,-.1)[label = below: $0$]{};
		\draw[thick](.1,0)--(.2,0);
			\fill[black](.3,-.1) circle (1pt);
		\fill[black](0.3,.1) circle (1pt);
		\node at (0.3,.1)[label = above: $\m $]{};
		\draw[thick](.4,0)--(.9,0);
		\fill[black](1,-.1) circle (1pt);
		\fill[black](1,.1) circle (1pt);
		\node at (1,-.1)[label = below: $1$]{};
			\draw[thick](1.1,0)--(1.2,0);		
		\fill[black](1.3,-.1) circle (1pt);
		\fill[black](1.3,.1) circle (1pt);
		\node at (1.3,.1)[label = above: $\m + 1$]{};
		\draw[thick](1.4,0)--(1.9,0);		
		\fill[black](2,-.1) circle (1pt);
		\fill[black](2,.1) circle (1pt);
		\node at (2,-.1)[label = below: $2$]{};
		\draw[thick,](2.1,0)--(2.2,0);	
			\fill[black](2.3,-.1) circle (1pt);
		\fill[black](2.3,.1) circle (1pt);
		\node at (2.3,.1)[label = above: $\m + 2$]{};
		\draw[thick,](2.4,0)--(2.9,0);		
		\fill[black](3,-.1) circle (1pt);
		\fill[black](3,.1) circle (1pt);
		\node at (3,-.1)[label = below: $3$]{};
			\draw[thick,](3.1,0)--(3.2,0);	
		\fill[black](3.3,-.1) circle (1pt);
		\fill[black](3.3,.1) circle (1pt);
		\node at (3.3,.1)[label = above: $\m + 3$]{};
		\draw[thick,->](3.4,0)--(3.9,0);		
	\end{tikzpicture}
	}
\noindent In this case, for each integer $n$, $A/(z+n)A$ is a non-split extension of two simple modules which we call $X^f_0\s{n}$ and $Y^f_0\s{n}$ and additionally, $A/(z+n+\m )$ is a non-split extensions of two simple modules which we call $X^f_{\m } \s{n}$ and $Y^f_{\m }\s{n}$. Each pair of these simple modules behaves in the same way as the pair $X \s{n}$ and $Y \s{n}$ in $\gr A_1$.

In the multiple root case, the picture of $\gr A(f)$ is the same as for $\gr A_1$.

 \centerline{
	\begin{tikzpicture}[ scale=1.6]	
			\node at (-1.7,.1)[label = above: {}]{};
		\draw[thick,<-](-3.9,0)--(-3.1,0);
		\fill[black](-3,-.1) circle (1pt);
		\fill[black](-3,.1) circle (1pt);
		\node at (-3,-.1)[label = below: $-3$]{};
		\draw[thick](-2.9,0)--(-2.1,0);
		\fill[black](-2,-.1) circle (1pt);
		\fill[black](-2,.1) circle (1pt);
		\node at (-2,-.1)[label = below: $-2$]{};
		\draw[thick](-1.9,0)--(-1.1,0);
		\fill[black](-1,-.1) circle (1pt);
		\fill[black](-1,.1) circle (1pt);
		\node at (-1,-.1)[label = below:$-1$]{};
		\draw[thick](-.9,0)--(-.1,0);
		\fill[black](0,-.1) circle (1pt);
		\fill[black](0,.1) circle (1pt);
		\node at (0,-.1)[label = below: $0$]{};
		\draw[thick](.1,0)--(.9,0);
		\fill[black](1,-.1) circle (1pt);
		\fill[black](1,.1) circle (1pt);
		\node at (1,-.1)[label = below: $1$]{};
		\draw[thick](1.1,0)--(1.9,0);		
		\fill[black](2,-.1) circle (1pt);
		\fill[black](2,.1) circle (1pt);
		\node at (2,-.1)[label = below: $2$]{};
		\draw[thick,](2.1,0)--(2.9,0);		
		\fill[black](3,-.1) circle (1pt);
		\fill[black](3,.1) circle (1pt);
		\node at (3,-.1)[label = below: $3$]{};
		\draw[thick,->](3.1,0)--(3.9,0);		
	\end{tikzpicture}
	}
	
\noindent For every integer $n$, $A/(z+n)A$ is a nonsplit extension of two simple modules, $X^f\s{n}$ and $Y^f\s{n}$. However, although the picture is the same, the category is not---the simple modules $X^f\s{n}$ and $Y^f\s{n}$ also have non-split self-extensions.

Finally, in the congruent root case, there exist finite-dimensional simple modules: the shifts of a module we call $Z^f$. For each integer $n$, $A/(z+n)A$ has a composition series consisting of $X^f\s{n}$, $Y^f\s{n}$, and $Z^f\s{n}$.

\centerline{
	\begin{tikzpicture}[ scale=1.6]	
			\node at (-1.7,.1)[label = above: {}]{};
		\draw[thick,<-](-3.9,0)--(-3.1,0);
		\fill[black](-3,-.1) circle (1pt);
		\fill[black](-3,0) circle (1pt);
		\fill[black](-3,.1) circle (1pt);
		\node at (-3,-.1)[label = below: $-3$]{};
		\draw[thick](-2.9,0)--(-2.1,0);
		\fill[black](-2,-.1) circle (1pt);
		\fill[black](-2,-0) circle (1pt);
		\fill[black](-2,.1) circle (1pt);
		\node at (-2,-.1)[label = below: $-2$]{};
		\draw[thick](-1.9,0)--(-1.1,0);
		\fill[black](-1,-.1) circle (1pt);
		\fill[black](-1,0) circle (1pt);
		\fill[black](-1,.1) circle (1pt);
		\node at (-1,-.1)[label = below:$-1$]{};
		\draw[thick](-.9,0)--(-.1,0);
		\fill[black](0,-.1) circle (1pt);
		\fill[black](0,0) circle (1pt);
		\fill[black](0,.1) circle (1pt);
		\node at (0,-.1)[label = below: $0$]{};
		\draw[thick](.1,0)--(.9,0);
		\fill[black](1,-.1) circle (1pt);
		\fill[black](1,0) circle (1pt);
		\fill[black](1,.1) circle (1pt);
		\node at (1,-.1)[label = below: $1$]{};
		\draw[thick](1.1,0)--(1.9,0);		
		\fill[black](2,-.1) circle (1pt);
		\fill[black](2,0) circle (1pt);
		\fill[black](2,.1) circle (1pt);
		\node at (2,-.1)[label = below: $2$]{};
		\draw[thick,](2.1,0)--(2.9,0);		
		\fill[black](3,-.1) circle (1pt);
		\fill[black](3,0) circle (1pt);
		\fill[black](3,.1) circle (1pt);
		\node at (3,-.1)[label = below: $3$]{};
		\draw[thick,->](3.1,0)--(3.9,0);		
	\end{tikzpicture}
	}

For every quadratic polynomial $f$, we construct autoequivalences of $\gr A(f)$ which are analogous to Sierra's involutions $\iota_n$. We determine the Picard group of $A(f)$, showing that $\Pic\left(\gr A(f) \right) \cong \Pic \left(\gr A_1 \right)$.

\begin{theorem}[Theorem~\ref{picthm}] For any quadratic polynomial $f \in \kk[z]$,
\[ \Pic(\gr A(f)) \cong \ZZ_{\fin} \rtimes D_{\infty}. 
\]
\end{theorem}

In \cite{smith}, Paul Smith took an interesting point of view on the first Weyl algebra $A_1$. Using Sierra's involutions, Smith constructed a commutative $\ZZ_\fin$-graded ring $C$ and showed that $\gr A_1 \equiv \gr (C, \ZZ_\fin)$. As a corollary, he showed that $\gr A_1$ is equivalent to the coherent sheaves on a certain quotient stack. This result suggests that it may be interesting to consider the geometry of noncommutative $\ZZ$-graded rings.

In a companion paper \cite{woncomm}, we use the autoequivalences of $\gr A(f)$ constructed in this paper to construct $\ZZ_\fin$-graded commutative rings $C(f)$ which are similar to Smith's ring $C$. Let $\qgr A(f)$ be the quotient category $\gr A(f)$ modulo its full subcategory of finite-dimensional modules.
\begin{theorem}[{\cite[Theorems~4.2, 4.10, 4.19]{woncomm}}] Let $f \in \kk[z]$ be quadratic. There is a commutative $\ZZ_\fin$-graded ring $C(f)$ such that
\[ \qgr A(f) \equiv \gr\left (C(f), \ZZ_\fin\right).
\]
\end{theorem}

This generalization of Smith's result provides an additional example of $\ZZ$-graded geometry. It would be interesting to consider other noncommutative $\ZZ$-graded rings. When is there a commutative $\Gamma$-graded ring with an equivalent graded module category? One could attempt the construction used in \cite{smith} and \cite{woncomm} to study other GWAs $R$. To do this, one first needs to understand the autoequivalences of $\gr R$. Therefore, one might generalize the results of this paper to determine $\Pic(\gr R)$. To start, one could consider GWAs with the same base ring $\kk[z]$ and automorphism $\sigma$ as studied in this paper. Many of the techniques should generalize to higher degree polynomials $f$. 

To close, we briefly summarize the contents of this paper. In section~\ref{sec:bg}, we establish notation and give background on GWAs and abelian categories. In sections~\ref{sec:simple} and \ref{sec:finlen}, we study the finite length modules of $\gr A(f)$. In section~\ref{sec:ideals}, we study the graded submodules of $\QgrAf$, the graded quotient ring of $A(f)$. We use these results in section~\ref{sec:proj}, to investigate the full subcategory of $\gr A(f)$ consisting of the rank one projective modules. In section~\ref{sec:func}, we develop some technical tools to allow us to define a functor on $\gr A$ by first defining it on a full subcategory of projective modules. Finally, in section~\ref{sec:pic}, we construct the autoequivalences of $\gr A(f)$ which are analogous to the $\iota_J$ of \cite{sierra} and compute the Picard group of $A(f)$. 

\begin{acknowledgments}The author was partially supported by NSF grant DMS-1201572. The research in this paper was completed while the author was a PhD student at the University of California, San Diego. In some cases, extra details may be found in the author's PhD thesis \cite{won}. The author is grateful to Daniel Rogalski for carefully reading drafts of this paper and suggesting many improvements, and to Sue Sierra and Paul Smith for helpful conversations.
\end{acknowledgments}

\section{Preliminaries}
\label{sec:bg}

\subsection{Graded rings and modules}
\label{sec:graded}
We begin by fixing basic definitions, terminology, and notation. We follow the convention that 0 is a natural number so $\NN = \ZZ_{\geq 0}$. We use the notation $\NN^+$ to denote the positive natural numbers.

In this paper, a \emph{graded ring} is a $\ZZ$-graded $\kk$-algebra. For a $\ZZ$-graded Ore domain $R$, we write $Q_{\mathrm{gr}}(R)$ for the graded quotient ring of $R$, the localization of $R$ at all nonzero homogeneous elements. The category of finitely generated $\ZZ$-graded right $R$-modules is denoted $\gr R$, while the category of all finitely generated right $R$-modules is denoted $\mod R$. 

For $\ZZ$-graded right $R$-modules $M$ and $N$, let $\uHom_R(M,N)_d$ denote the graded $R$-module homomorphisms of degree $d$ from $M$ to $N$ and define
\[ \uHom_R(M,N) = \bigoplus_{d \in \ZZ} \uHom_R(M,N)_d.
\] 
The morphisms in $\gr R$ are the graded homomorphisms of degree $0$. We denote these homomorphisms by 
\[ \Hom_{\gr R}(M,N) = \uHom_R(M,N)_0.
\]
We write $\uExt_R$ and $\Ext_{\gr R}$ for the derived functors of $\uHom_R$ and $\Hom_{\gr R}$, respectively.

For a graded $\kk$-algebra $R$, the \emph{shift functor} on $\gr R$ sends a graded right module $M$ to the new module $M \langle 1 \rangle = \bigoplus_{j \in \ZZ} M \langle 1 \rangle _j$, defined by $M \langle 1 \rangle _j = M_{j-1}$. We write this functor as $\SSS_R: M \mapsto M \s{1}$. Similarly, $M\s{i}_j = M_{j-i}$. This is in keeping with the convention of Sierra in \cite{sierra}, although we warn that this is the opposite of the standard convention. For a graded $\kk$-vector space $V$, we use the same notation to refer to the shift of grading on $V$: $V \langle i \rangle_j = V_{j-i}$. For right $R$-modules $M$ and $M'$, note that as graded vector spaces, 
\[\uHom_R(M\s{d}, M' \s{d'}) \cong \uHom_R(M,M') \s{d' - d}.\]

Given two categories $\CCC$ and $\DDD$, a covariant functor $\FFF: \CCC \sra \DDD$ is called an \emph{equivalence of categories} if there is a covariant functor $\GGG: \DDD \sra \CCC$ such that $\GGG \circ \FFF \cong \Id_{\CCC}$ and $\FFF \circ \GGG \cong \Id_{\DDD}$. We say that $\CCC$ and $\DDD$ are \emph{equivalent} and write $\CCC \equiv \DDD$. An equivalence of categories $\FFF: \CCC \sra \CCC$ is called an \emph{autoequivalence of $\CCC$}. Given a category of $\Gamma$-graded modules, $\gr (R, \Gamma)$, let $\Aut(\gr (R,\Gamma))$ be the group of autoequivalences of $\gr (R,\Gamma)$ with operation composition. Denote by $\sim$ the equivalence relation on $\Aut(\gr (R, \Gamma))$ given by natural isomorphism. We define the \emph{Picard group of $\gr (R, \Gamma)$}
\[\Pic(\gr (R, \Gamma)) = \Aut(\gr(R,\Gamma)) / \sim. \]

Following \cite{sierra2}, we define a \textit{$\ZZ$-algebra} $R$ to be a $\kk$-algebra without 1, with a $(\ZZ \times \ZZ)$-graded $\kk$-vector space decomposition
\[ R = \bigoplus_{i, j \in \ZZ} R_{i,j}
\]
such that for any $i,j,k \in \ZZ$, $R_{i,j} R_{j,k} \subseteq R_{i,k}$ and if $j \neq j'$ then $R_{i,j } R_{j', k} = 0$. Additionally, we require that each of the subrings $R_{i,i}$ has a unit $1_{i}$ which acts as a left identity on all $R_{i,j}$ and right identity on all $R_{j,i}$. If $R$ and $S$ are $\ZZ$-algebras, we say that a $\kk$-algebra homomorphism $\varphi: R \sra S$ is \textit{graded of degree $d$} if $\varphi(R_{i,j})\subseteq S_{i+d, j+d}$. A \textit{$\ZZ$-algebra isomorphism} is a degree $0$ $\kk$-algebra isomorphism $\varphi: R \sra S$ such that $\varphi(1_i) = 1_i$ for all $i \in \ZZ$. 

For a $\ZZ$-graded ring $R$, we define the \textit{$\ZZ$-algebra associated to $R$}
\[ \ol{R} = \bigoplus_{i,j \in \ZZ} \ol{R}_{i,j}
\]
where $\ol{R}_{i,j} = R_{j-i}$. A degree $0$ automorphism of $\ol{R}$ is called \textit{inner} (see \cite[Theorem 3.10]{sierra}) if for all $m,n \in \ZZ$, there exist $g_m \in \ol{R}_{m,m}$ and $h_n \in \ol{R}_{n,n}$ such that for all $w \in \ol{R}_{m,n}$, 
\[ \gamma(w) = g_m w h_n.
\]

In \cite{sierra} and \cite{sierra2}, Sierra studies the relationship between $\ol{R}$ and $\gr R$. In particular, Sierra proves the following:

\begin{theorem}[Sierra, {\cite[Theorem 3.6]{sierra2}}] \label{zalgebra}Let $R$ and $S$ be $\ZZ$-graded $\kk$-algebras. The following are equivalent:
\begin{enumerate}
\item The $\ZZ$-algebras $\ol{R}$ and $\ol{S}$ are isomorphic via a degree-preserving map.
\item \label{twist} There is an equivalence of categories $\Phi: \gr R \sra \gr S$ such that $\Phi(R \s{n}) \cong S \s{n}$ for all $n \in \ZZ$.
\end{enumerate}
\end{theorem}

A functor $\Phi$ satisfying condition \ref{twist} in Theorem~\ref{zalgebra} above is called a \textit{twist functor}. Sierra further gives the following result on twist functors.

\begin{theorem}[Sierra, {\cite[Corollary 3.11]{sierra}}] \label{twistfunctors} Let $R$ be a $\ZZ$-graded ring and suppose that all automorphisms of $\ol{R}$ of degree $0$ are inner. If $\Phi: \gr R \sra \gr R$ is a twist functor, then $\Phi$ is naturally isomorphic to $\Id_{\gr R}$.
\end{theorem}

\subsection{Generalized Weyl algebras}
Fix $f \in \kk[z]$, let $\sigma: \kk[z] \sra \kk[z]$ be the automorphism given by $\sigma(z) = z+1$, and let 
\[A(f) = \frac{\kk[z] \langle x, y \rangle}{\begin{pmatrix}xz = \sigma(z)x & yz = \sigma^{-1}(z)y \\ xy = f & yx = \sigma^{-1}(f)\end{pmatrix}}.
\]
Then $A(f) = \kk[z]\left(\sigma, f\right)$ is a \emph{generalized Weyl algebra of degree 1} with base ring $\kk[z]$, defining element $f$ and defining automorphism $\sigma$. Generalized Weyl algebras were introduced by Vladimir Bavula, who studied rings of the form $A(f)$ for $f$ of arbitrary degree \cite{bav, bavjordan}. Timothy Hodges studied the same rings under the name noncommutative deformations of type-A Kleinian singularities \cite{hodges}. By results in \cite{bav}, for all $f$, $A(f)$ is a noncommutative noetherian domain of Krull dimension $1$. 

By a theorem of Bavula and Jordan, $A(f) \cong A(g)$ if and only if $f(z) = \eta g( \tau \pm z)$ for some $\eta, \tau \in \kk$ with $\eta \neq 0$ \cite[Theorem 3.28]{bavjordan}. Hence, by adjusting $\eta$ we may assume $f$ is a monic polynomial and by adjusting $\tau$ we may assume that $0$ is a root of $f$. We may also assume that $0$ is the largest integer root of $f$. Results of Bavula and Hodges show that the properties of $A(f)$ are determined by the distance between the roots of $f$. In general we say that two distinct roots, $\lambda$ and $\mu$ are \emph{congruent} if $\lambda - \mu \in \ZZ$. The global dimension of $A(f)$ depends only on whether $f$ has multiple or congruent roots, as follows.
\begin{theorem}[Bavula {\cite[Theorem 5]{bav}} and Hodges {\cite[Theorem 4.4]{hodges}}] \label{gldimA} The global dimension of $A$ is equal to
\[ \gldim A = \begin{cases}\infty, & \text{if $f$ has at least one multiple root} \\ 2, & \text{if $f$ has no multiple roots but has congruent roots}; \\ 1, & \text{if $f$ has neither multiple nor congruent roots.} \end{cases}
\]
\end{theorem}

In this paper, we study generalized Weyl algebras $A(f)$ for quadratic polynomials $f$. Without loss of generality, $f = z(z+\m)$. When $\m = 0$, since $f$ has a multiple root, we say we are in the \emph{multiple root case}. When $\m \in \NNp $, we say we are in the \emph{congruent root case}. Finally, when $\m \in \kk \setminus \ZZ$, we say that $f$ has distinct non-congruent roots and refer to this case as the \emph{non-congruent root case}.

Like the first Weyl algebra, the rings $A(f)$ are naturally $\ZZ$-graded by letting $\deg x = 1$, $\deg y = -1$, $\deg z = 0$. Note that every graded right $A(f)$-module is actually a graded $(\kk[z],A(f))$-bimodule; for any $\varphi \in \kk[z]$, we define its left action on a right $A(f)$-module $M$ by
\[ \varphi \cdot m = m\sigma^{-i}(\varphi)
\]
for any $m \in M_i$. This gives $M$ a bimodule structure by the relations $xz = \sigma(z)x$ and $yz = \sigma^{-1}(z)y$. 

Bavula and Jordan \cite{bavjordan} call a polynomial $g(z) \in \kk[z]$ \emph{reflective} if there exists some $\beta \in \kk$ such that $g(\beta - z) = g(z)$. They observe that every quadratic polynomial is reflective. Indeed, if $f$ is quadratic, there exists an outer automorphism $\omega$ of $A(f)$ such that $\omega(x) = y$, $\omega(y) = x$, and $\omega(z) = 1 - \m - z$ which reverses the grading on $A$. More specifically, there is a group automorphism of $\ZZ$, and the $\kk$-algebra automorphism $\omega$ respects this automorphism. The group automorphism, which we denote $\bar{\omega}$, is given by negation and $\omega(A_n) = A_{\bar{\omega}(n)}$. Together, we call $(\omega, \bar{\omega})$ a \emph{$\bar{\omega}$-twisted graded ring automorphism} of $A$.

For any group automorphism $\bar{\theta}$ of $\ZZ$, any $\bar{\theta}$-twisted graded ring automorphism $(\theta, \bar{\theta})$ of $A$ induces an autoequivalence of $\gr A$, denoted $\theta_*$. Given a module $M \in \gr A$, $\theta_* M$ is defined to be $M$ with the grading $(\theta_* M)_n = M_{\bar{\theta}(n)}$. We write $\theta_* m$ to regard the element $m \in M$ as an element of $\theta_* M$. The action of an element $a \in A$ on an element $\theta_* m \in \theta_* M$ is given by 
\[ \theta_* m \cdot a = \theta_*( m \theta(a)).
\]

We suppress the asterisk in the subscript and simply refer to the autoequivalence induced by $(\omega, \bar{\omega})$ as $\omega$. We also define $\omega$ on a graded $\kk$-vector space as the functor that reverses grading, i.e. $(\omega V)_n = V_{-n}$. Then, as Sierra notes in \cite[(4.2)]{sierra}, if $M$ and $N$ are right $A$-modules, then $\omega$ gives isomorphisms of graded $\kk$-vector spaces
\begin{align} \begin{split} \uHom_A(\omega M, \omega N) \cong \omega \uHom_A(M, N) \\
\uExt_A^1(\omega M , \omega N) \cong \omega \uExt_A^1(M, N). \label{omegaext}
\end{split}
\end{align}

\section{The graded module category $\gr A$}
\subsection{The simple modules of $\gr A$}
\label{sec:simple}
We first describe the simple modules of $\gr A(f)$.

\begin{lemma}
\label{simples} Let $f = z(z+\m )$. 
\begin{enumerate}
\item If $\m = 0$, then up to graded isomorphism the graded simple $A(f)$-modules are:
\begin{itemize}
\item $X^f = A/(x,z)A$ and its shifts $X^f \s{n}$ for each $n \in \ZZ$;
\item $Y^f = \left(A/(y,z-1)A\right) \s{1}$ and its shifts $Y^f \s{n}$ for each $n \in \ZZ$;
\item $M^f_\lambda = A/(z+\lambda)A$ for each $\lambda \in \kk \setminus \ZZ $.
\end{itemize}

\item If $\m \in \NNp$, then up to graded isomorphism the graded simple $A(f)$-modules are:
\begin{itemize}
\item $X^f = \left(A/(x,z+\m )A\right) \s{-\m }$ and its shifts $X^f \s{n}$ for each $n \in \ZZ$;
\item $Y^f = \left(A/(y,z-1)A\right) \s{1}$ and its shifts $Y^f \s{n}$ for each $n \in \ZZ$;
\item $Z^f = A/(y^\m , x, z)A$ and its shift $Z^f \s{n}$ for each $n \in \ZZ$;
\item $M^f_\lambda = A/(z+\lambda)A$ for each $\lambda \in \kk \setminus \ZZ $.
\end{itemize}

\item If $\m \in \kk \setminus \ZZ$, then up to graded isomorphism the graded simple $A(f)$-modules are:
\begin{itemize}
\item $X^f_0 = A/(x,z)A$ and its shifts $X^f_0\s{n}$ for each $n \in \ZZ$;
\item $Y^f_0 = \left(A/(y,z-1)A \right) \s{1}$ and its shifts $Y^f_0\s{n}$ for each $n \in \ZZ$;
\item $X^f_\m = A/(x,z+\m )A$ and its shifts $X^f_\m \s{n}$ for each $n \in \ZZ$;
\item $Y^f_\m = \left(A/(y,z+\m -1)A\right) \s{1}$ and its shifts $Y^f_\m \s{n}$ for each $n \in \ZZ$;
\item $M^f_\lambda = A/(z+\lambda)A$ for each $\lambda \in \kk \setminus (\ZZ \cup \ZZ +\m )$.
\end{itemize}
\end{enumerate}
\end{lemma}
\begin{proof}We do, as an example, the case that $\m = 0$. The other cases follow similarly from \cite[\S 3]{bav}. In \cite[\S 3]{bav}, Bavula studies the simple \emph{$\kk[z]$-torsion} $A$-modules, that is, modules for which $\tor (M) := \{ m\in M \mid m\cdot g = 0 \mbox{ for some } 0 \neq g \in \kk[z]\}$ is equal to $M$. Every graded simple right $A$- module is isomorphic to $A/I$ for some homogeneous right ideal $I$ of $A$. Further, every homogeneous element of $A$ can be written as $g x^i$ or $g y^i$ for some $g \in \kk[z]$ and $i \in \NN$. Hence, for every element $a$ of $A/I$, there exists $h \in \kk[z]$ such that $a \cdot h = 0$, so $A/I$ is $\kk[z]$-torsion. By \cite[Theorem 3.2]{bav}, up to ungraded isomorphism, the simple $\kk[z]$-torsion $A$-modules are
\begin{itemize} 
\item $A/(x,z)A$
\item $A/(y,z-1)A$
\item One module $A/(z+\lambda)A$ for each coset of $\kk/ \ZZ$.
\end{itemize}
For each $\lambda \in \kk \setminus \ZZ$, observe that $M^f_{\lambda + 1} \cong M^f_\lambda \s{1}$ via the isomorphism mapping $\bar{1}$ to $\bar{y}$. Further, by Bavula's theorem, if $\lambda - \mu \notin \ZZ$, then $M^f_\lambda$ and $M^f_\mu$ are not even ungraded isomorphic. Hence, if $\lambda \neq \mu$ then $M^f_\lambda \not \cong M^f_\mu$ in $\gr A$. Finally, we see that there are no other graded isomorphisms between any shift of $X^f$, $Y^f$, or $M^f_\lambda$ simply by looking at the degrees in which these modules are nonzero (see Remark~\ref{degzero} below). Hence, we conclude that the graded isomorphism classes of graded simples correspond to the shifts $X^f\s{n}$, $Y^f \s{n}$, and one module $M^f_\lambda$ for each element of $\kk \setminus \ZZ$.
\end{proof}

When there is no danger of confusion, we will make two changes in notation for convenience. When it is clear which case we are in (multiple, congruent, or distinct roots), we will suppress the superscript on graded simple modules and refer to them as $X$, $Y$, $Z$, and $M_\lambda$. Also, for a right ideal $I$, we will often refer to the element $a + I \in A/I$ simply as $a$.

\begin{remark}\label{degzero} We also remark that for each integer $n$ and each simple module $S$, $\dim_{\kk} S \leq 1$. We explicitly give the degrees in which each simple module is nonzero. Additionally, by using the explicit description of the simple modules as quotients of $A$, we can also determine the action of the autoequivalence $\omega$ on the graded simple modules. In all cases, for $\lambda \in \kk \setminus (\ZZ \cup \ZZ + \m )$, $\dim_{\kk} (M_\lambda)_n = 1$ for all $n$ and $\omega(M_\lambda) = M_\mu$ for some $\mu \in \kk \setminus (\ZZ \cup \ZZ + \m )$.

\begin{enumerate}
\item If $\m = 0$, then
\begin{itemize}
\item $\dim_{\kk} X_n = 1$ if and only if $n \leq 0$,
\item $\dim_{\kk} Y_n = 1$ if and only if $n > 0$,
\item $\omega(X \s{n}) \cong Y \s{-n-1}$ and $\omega(Y \s{n}) \cong X \s{-n-1}$.
\end{itemize}

\item If $\m \in \NNp$, then
\begin{itemize}
\item $\dim_{\kk} X_n = 1$ if and only if $n \leq -\m $,
\item $\dim_{\kk} Y_n = 1$ if and only if $n > 0$,
\item $\dim_{\kk} Z_n = 1$ if and only if $-\m < n \leq 0$,
\item $\omega(X \s{n}) \cong Y \s{\m -n-1}$, $\omega(Y \s{n}) \cong Y \s{\m -n-1}$, and $\omega(Z \s{n}) \cong Z \s{\m -n -1}$.
\end{itemize}

\item If $\m \in \kk \setminus \ZZ$, then
\begin{itemize}
\item $\dim_{\kk} (X_0)_n = \dim_{\kk}(X_{\m })_n = 1$ if and only if $n \leq 0$,
\item $\dim_{\kk} (Y_0)_n = \dim_{\kk}(Y_{\m })_n = 1$ if and only if $n > 0$,
\item $\omega(X_0 \s{n}) \cong Y_{\m } \s{ -n-1}$, $\omega(Y_0 \s{n}) \cong Y_{\m } \s{-n-1}$, 
\item $\omega(X_{\m } \s{n}) \cong Y_{0} \s{ -n-1}$, and $\omega(Y_{\m } \s{n}) \cong Y_{0} \s{-n-1}$.
\end{itemize}
\end{enumerate}
\end{remark}

\begin{lemma} \label{csupport} Let $\m \in \NN$ and let $n \in \ZZ$. Then $(z+n)(X\s{n}) = (z+n)(Y \s{n}) = 0$. If $\m \neq 0$, then $(z+n)(Z \s{n}) = 0$. As graded left $\kk[z]$-modules, we have
\[ (A/zA)\s{n} \cong (A/(z-1)A)\s{n+1}\cong \bigoplus_{j \in \ZZ} \frac{\kk[z]}{(z+n)}.
\]
\end{lemma}
\begin{proof} Recall that there is a left action of $\kk[z]$ on an $A$-module. If $\deg m = i$ then for $p \in \kk[z]$, $p \cdot m = m \sigma^{-i}(p)$. This result follows from the fact that in $X \s{n} = \left(A/(x, z+\m)A\right) \s{-\m+n}$, $\deg 1 = n-\m$ and
\[ (z+n) \cdot 1 = 1\sigma^{\m-n}(z+n) = z+\m = 0.
\]
Similarly, in $Y \s{n}$, $\deg 1 = n+1$ so
\[ (z+n) \cdot 1 = 1 \sigma^{-n-1}(z+n) = z-1 = 0,
\]
and in $Z \s{n}$, $\deg 1 = 0$ and
\[ (z+n) \cdot 1 = z = 0.
\]

For each $j\geq n$, $\left((A/zA) \s{n}\right)_j$ is generated as a left $\kk[z]$ module by $x^{j-n}$. Further, for all $g \in \kk[z]$, $g \cdot x^{j-n} = x^{j-n} \sigma^{n-j}(g) = \sigma^{n}(g)x^{j-n}$. Hence, as a left $\kk[z]$-module the annihilator of $((A/zA)\s{n})_j$ is given exactly by the ideal $(z+n)$. For $j < n$, $\left((A/zA) \s{n}\right)_j$ is generated as a left $\kk[z]$ module by $y^{n-j}$. By a similar argument, the annihilator is given exactly by the ideal $(z+n)$ so 
\[ (A/zA)\s{n} \cong \bigoplus_{j \in \ZZ} \frac{\kk[z]}{(z+n)}.
\]
The same proof shows that
\[(A/(z-1)A)\s{n+1}\cong \bigoplus_{j \in \ZZ} \frac{\kk[z]}{(z+n)}. \qedhere
\] \end{proof}

\begin{lemma} \label{dsupport} Let $\m \in \kk \setminus \ZZ$ and let $n \in \ZZ$. Then $(z+n)(X_0\s{n}) = 0 = (z+n)(Y_0 \s{n})$ and $(z + \m + n)(X_\m \s{n}) = 0 = (z + \m + n)(Y_\m \s{n})$. As graded left $\kk[z]$-modules, we have
\[ (A/zA)\s{n} \cong (A/(z-1)A)\s{n+1} \cong \bigoplus_{j \in \ZZ} \frac{\kk[z]}{(z+n)} \mbox{ and}
\]
\[ (A/(z+\m )A)\s{n} \cong (A/(z+\m -1)A)\s{n+1} \cong \bigoplus_{j \in \ZZ} \frac{\kk[z]}{(z+\m +n)}.
\]
\end{lemma}
\begin{proof}This follows from the same proof as that of Lemma~\ref{csupport}.
\end{proof}

Since $M_\lambda \cong \bigoplus_{j\in \ZZ} \kk[z]/(z+\lambda)$ as a left $\kk[z]$-module, when combined with the previous lemmas, any finite length graded $A$-module, when considered as a left $\kk[z]$-module, is supported at finitely many $\kk$-points of $\Spec \kk[z]$. We restate \cite[Definition 4.9]{sierra}. 

\begin{definition} If $M$ is a graded $A$-module of finite length, define the \textit{support} of $M$, $\Supp M$, to be the support of $M$ as a left $\kk[z]$-module. We are interested in the cases when $\Supp M \subset \ZZ$ or $\Supp M \subset \ZZ - \m $. When $\Supp M \subset \ZZ$, we say that $M$ is \emph{integrally supported}. We are also interested in cases when $\Supp M = \{ n \}$ or $\Supp M = \{n - \m\}$ for some $n \in \ZZ$ (we say $M$ is \emph{simply supported at $n$ or $n - \m$}). 
\end{definition}

Lemma~\ref{csupport} shows that in the case that $\m \in \NN$, $X\s{n}$ and $Y\s{n}$ are the unique simples supported at $-n$. In the case that $\m \notin \ZZ$, Lemma~\ref{dsupport} shows that $X_0\s{n}$ and $Y_0\s{n}$ are the unique simples supported at $-n$ and $X_\m \s{n}$ and $Y_\m \s{n}$ are the unique simples supported at $-(n+\m )$. For $\lambda \in \kk \setminus ( \ZZ \cup \ZZ+ \m )$, the simple module $M_\lambda$ is the unique simple supported at $-\lambda$.

\subsection{The structure constants of a graded submodule of the graded quotient ring of $A$}
\label{sec:ideals}

We seek to understand the rank one projective modules of $\gr A(f)$. Since $A(f)$ is noetherian, $A(f)$ is an Ore domain, so we can construct $\QgrAf$, the graded quotient ring of $A(f)$. Every homogeneous element of $A(f)$ can be written as $x^i g(z)$ or $y^i g(z)$ for some $i \geq 0$ and some $g(z) \in \kk[z]$. Since in the graded quotient ring $y = x^{-1} f$, we see that $\QgrAf$ embeds in the skew Laurent polynomial ring $\kk(z)[x,x^{-1}; \sigma]$. Finally, since every element of $\kk(z)[x,x^{-1}; \sigma]$ can be written as a quotient of elements of $A(f)$, therefore $\QgrAf = \kk(z)[x,x^{-1};\sigma]$.

To understand the rank one projective modules, we begin by considering all submodules of $\QgrAf$. Let $I$ be a finitely generated graded right $A$-submodule of $\QgrAf$. Recall that $\sigma(z) = z+1$ and every graded right $A$-module is a graded left $\kk[z]$-module with
\[ \varphi \cdot m = m \cdot \sigma^{-i}(\varphi)
\]
for $\varphi \in \kk[z]$ and $\deg m = i$. We will examine $I$ as a graded left $\kk[z]$-submodule of $\QgrA$. As a left $\kk[z]$-module,
\[ \QgrA \cong \bigoplus_{i \in \ZZ} \kk(z) x^i.
\]

Suppose $I$ is generated as an $A$-module by the homogeneous generators $m_1, \ldots, m_r$, with $\deg m_i = d_i$. Then for each $n \in \ZZ$, $I_n = \sum m_i A_{n-d_i}$ where $I_n$ is the degree $n$ graded component of $I$. Since each graded component of $A$ is finitely generated as a left $\kk[z]$-module, so is $I_n$. If we clear denominators and use the fact that $\kk[z]$ is a PID, we deduce that $I_n$ is generated as a left $\kk[z]$-module by a single element $a_n x^n$ where $a_n \in \kk(z)$. Denote by $(a_n)$ the left $\kk[z]$-submodule of $\kk(z)$ generated by $a_n$. Then
\[ I = \bigoplus_{i\in \ZZ}(a_i)x^i.
\]
Because $I$ is a right $A$-submodule, we have for each $i \in \ZZ$
\begin{align*} &(a_i)x^i \cdot x \subseteq (a_{i+1}) x^{i+1} \quad \mbox { and } \\
&(a_{i+1})x^{i+1} \cdot y = (a_{i+1}) \sigma^{i}(f) x^{i} \subseteq (a_{i})x^{i}.
\end{align*}
Therefore, for each $i \in \ZZ$, we have $(a_i) \subseteq(a_{i+1})$ and $(a_{i+1} \sigma^i(f)) \subseteq (a_i)$. Define $c_i = a_i a_{i+1}^{-1}$. We then have $1 \mid c_i$ and $c_i \mid \sigma^{i}(f)$, so $c_i \in \kk[z]$. By multiplying by an appropriate element of $\kk$, we assume that $c_i$ is monic so
\[c_i \in \{ 1, \sigma^{i}(z) , \sigma^{i}(z+\m), \sigma^i(f)\}.
\]
\begin{definition}We call the elements of this sequence $\{c_i\}_{i \in \ZZ}$ the \emph{structure constants} of $I$. The lemma below shows that a finitely generated graded right $A$-submodule of $\QgrA$ is determined up to graded isomorphism by its structure constants.
\end{definition}

As an example, we compute the structure constants of the ring $A$.

\begin{example}\label{scA} For $n \in \ZZ$, $A_n$ is generated as a left $\kk[z]$-module by $x^n$ when $n \geq 0$ and $y^{-n}$ when $n < 0$. Also, for $n > 0$, $y^nx^n = \sigma^{-1}(f) \cdots \sigma^{-n}(f)$, so as a graded left $\kk[z]$-module,
\[ A = \bigoplus_{i \in \ZZ} (a_i)x^i \mbox{ with } a_i = \begin{cases}1, & i \geq 0, \\ \sigma^{-1}(f) \cdots \sigma^{i}(f), & i < 0. \end{cases}
\]
The structure constants $\{c_i\}$ of $A$ are therefore given by
\[ c_i = \begin{cases} 1 & i \geq 0, \\
\sigma^{i}(f) & i < 0.
\end{cases}
\]
\end{example}

\begin{lemma} \label{sciso} Let $I$ and $J$ be finitely generated graded submodules of $\QgrA$ with structure constants $\{c_i\}$ and $\{d_i\}$, respectively. Then $I \cong J$ as graded right $A$-modules if and only if $c_i = d_i$ for all $i \in \ZZ$.
\end{lemma}
\begin{proof} As argued above there exist $\{a_i\}, \{b_i\} \subseteq \kk(z)$ such that
\[ I = \bigoplus_{i \in \ZZ} (a_i) x^i \quad \mbox{ and } \quad J = \bigoplus_{i \in \ZZ} (b_i)x^i.
\]
Then by definition, for each $ i \in \ZZ$, $c_i = a_i a_{i+1}^{-1}$ and $d_i = b_i b_{i+1}^{-1}$. Let $g = a_0^{-1}b_0 \in \kk(z)$. If $c_i = d_i$ for all $i \in \ZZ$, then $(a_i g) = (b_i)$ for all $i$. Hence, $I \cong J$ via left multiplication by $g$.

Conversely, suppose $\varphi: I \sra J$ is a graded isomorphism of $A$-modules. Since, for each $i \in \ZZ$, $\varphi: I_i \sra J_i$ is an isomorphism, we must have, up to a scalar in $\kk^\times$, $\varphi(a_ix^i) = b_i x^i$. Then, up to a scalar,
\begin{align*} b_ix^{i+1} &= \varphi(a_i x^i) x = \varphi (a_i x^{i+1}) = \varphi (a_{i+1}c_i x^{i+1}) = \varphi(a_{i+1} x^{i+1} \sigma^{-(i+1)}(c_i)) \\
&= \varphi(a_{i+1} x^{i+1}) \sigma^{-(i+1)}(c_i) = b_{i+1}x^{i+1} \sigma^{-(i+1)}(c_i) = b_{i+1} c_i x^{i+1},
\end{align*}
so up to a scalar $d_i = b_ib_{i+1}^{-1} = c_i$ for all $i$. Since we assumed that structure constants were monic, $c_i = d_i$ for all $i \in \ZZ$.
\end{proof}

\begin{lemma}\label{cfgsc} Suppose $I = \bigoplus_{i\in \ZZ} (b_i) x^i$ is a finitely generated graded right $A$-submodule of $\QgrA$ with structure constants $\{c_i\}$. Then for $n \gg 0$, $c_n = 1$ and $c_{-n} = \sigma^{-n}(f)$.
\end{lemma} 
\begin{proof} Since $I$ is finitely generated as an $A$-module, if $n \in \ZZ$ is greater than the highest degree of all generators, we have $I_n \cdot x = I_{n+1}$. Then $(b_n) = (b_{n+1})$ so $c_n = 1$. On the other hand, if $n$ is less than the least degree of all generators, then $I_{n} \cdot y = I_{n-1}$. That is, 
\[(b_{n-1})x^{n-1} = (b_n)x^ny = (b_n)x^{n-1} f = (b_n)\sigma^{n-1}(f) x^{n-1},
\]
so $c_{n-1} = \sigma^{n-1}(f)$. Hence, for $n \gg 0$, $c_n = 1$ and $c_{-n} = \sigma^{-n}(f)$.
\end{proof}

We remark that for any choice $\{c_i\}_{i \in \ZZ}$ satisfying (i) for each $n$, ${c_n \in \{ 1, \sigma^{n}(z) , \sigma^{n}(z+\m), \sigma^n(f)\}}$ and (ii) $c_n = 1$ and $c_{-n} = \sigma^{-n}(f)$ for $n \gg 0$, we can construct a module with structure constants $\{c_i\}$. Let $b_0 = 1$ and for all integers $i$, define $b_i$ such that $b_i b_{i+1}^{-1} = c_i$. Let $I = \bigoplus_{i \in \ZZ} (b_i)x^i$. Since $b_i x^i \cdot x \in (b_{i+1})x^{i+1}$ and $b_i x^i \cdot y \in (b_{i-1})x^{i-1}$, therefore $I$ is a graded submodule of $\QgrA$. Further, there exists $N \in \NN$ such that for all $n \geq N$, $c_n = 1$ and $c_{-n} = \sigma^{-n}(f)$. Thus, the elements $\{b_ix^i \mid -N \leq i \leq N \}$ generate $I$ as an $A$-module, so $I$ is a finitely generated graded right $A$-submodule of $\QgrA$.

Hence, isomorphism classes of finitely generated graded right $A$-submodules of $\QgrA$ are in bijection with such sequences of structure constants, $\{c_i\}$. One reason taking this point of view is useful is that we can now state properties of a graded submodule $I \subseteq \QgrA$ in terms of its structure constants and vice versa. First, we show that the simple factors of $I$ are determined by its structure constants. We have the two following lemmas, one in the case that $f$ has congruent or multiple roots ($\m \in \NN$) and one in the case that the roots of $f$ are distinct ($\m \in \kk \setminus \ZZ$).

\begin{lemma}\label{cssc} Let $\m \in \NN$. Let $I = \bigoplus_{i\in Z} (a_i)x^i$ be a finitely generated graded right $A$-submodule of $\QgrA$ with structure constants $\{c_i\}$. Then
\begin{enumerate} 
\item $I$ surjects onto $X\s{n}$ if and only if $c_{n-\m} \in \{ 1, \sigma^{n-\m}(z) \}$,
\item $I$ surjects onto $Y \s{n}$ if and only if $c_{n} \in \{ \sigma^{n}(z), \sigma^{n}(f) \}$,
\item If $\m > 0$, $I$ surjects onto $Z \s{n}$ if and only if $c_{n-\m} \in \{ \sigma^{n-\m}(z+\m), \sigma^{n-\m}(f)\}$ and $c_n \in \{1, \sigma^{n}(z+\m)\}$.
\end{enumerate}
\end{lemma}
\begin{proof}
Suppose $I$ surjects onto $Z\s{n}$. Then there exists a graded submodule $J = \bigoplus_{i \in \ZZ} (b_i)x^i$ of $I$ such that
\[0 \ra J \ra I \ra Z\s{n} \ra 0
\]
is a short exact sequence. Let $\{ d_i \}$ be the structure constants of $J$.

Because $Z\s{n}_i = 0$ for all $i > n$ and all $i \leq n -\m $, we must have $b_i = a_i$ for all $i > n$ and all $i \leq n -\m $. Now by Lemma~\ref{csupport}, as a left $\kk[z]$-module, for all $n - \m < i \leq n$, $Z \s{n}_i \cong \kk[z]/(z+n)$. Hence, $b_i = (z+n)a_i$ for all $n -\m < i \leq n$. Therefore, for all $i \neq n, n-\m$, $c_i = d_i$. However, $(z+n) d_{n - \m } = c_{n- \m }$ and $d_{n} = (z+n)c_{n}$. Since we know $d_{n - \m } \in \{ 1, \sigma^{n-\m}(z), \sigma^{n-\m}(z+\m), \sigma^{n-\m}(f)\}$, this forces $c_{n - \m } \in \{\sigma^{n-\m }(z+ \m ), \sigma^{n- \m }(f)\}$. Similarly, $c_{n} \in \{ 1, \sigma^{n}(z + \m ) \}$.

Conversely, if $c_{n-\m} \in \{ \sigma^{n-\m}(z+\m), \sigma^{n-\m}(f)\}$ and $c_n \in \{1, \sigma^{n}(z+\m)\}$, then we can construct $J = \bigoplus_{i \in \ZZ} (b_i)\kk[z] \subseteq I$ by setting $b_i = a_i$ for all $i > n$ and $i \leq n - \m $ and $b_i = (z+n)a_i$ for all $n - \m < i \leq 0$. If we define $d_i = b_ib_{i+1}^{-1}$ to be monic (by multiplying by the appropriate scalar), then because $c_{n-\m} \in \{ \sigma^{n-\m}(z+\m), \sigma^{n-\m}(f)\}$ and $c_n \in \{1, \sigma^{n}(z+\m)\}$, therefore $d_{n - \m} \in \{1, \sigma^{n- \m}(z+\m ) \}$ and $d_n \in \{ \sigma^{n}(z) , \sigma^{n}(f)\}$. For all other integers $i$, $d_i = c_i$. Therefore, the $\{d_i\}$ are the structure constants of a finitely generated graded submodule of $\QgrA$ which is isomorphic to $J$. By \cite[Theorem 2.1]{bav}, $A$ is $1$-critical. Hence, the factor module $I/J$ has finite length. Also, $I/J$ is simply supported at $-n$, and is nonzero only in degrees $n - \m < i \leq 0$. Hence, $I/J \cong Z \s{n}$ so $I$ surjects onto $Z\s{n}$.

The cases of $I$ surjecting onto $X\s{n}$ or $Y\s{n}$ are similar but give a condition on just one structure constant each.
\end{proof}

\begin{lemma}\label{dssc} Let $\m \in \kk \setminus \ZZ$. Let $I = \bigoplus_{i\in Z} (a_i)x^i$ be a finitely generated graded right $A$-submodule of $\QgrA$ with structure constants $\{c_i\}$. Then
\begin{enumerate} 
\item $I$ surjects onto $X^0\s{n}$ if and only if $c_{n} \in \{ 1, \sigma^{n}(z+\m) \}$,
\item $I$ surjects onto $X^\m\s{n}$ if and only if $c_{n} \in \{ 1, \sigma^{n}(z) \}$,
\item $I$ surjects onto $Y^0 \s{n}$ if and only if $c_{n} \in \{ \sigma^{n}(z), \sigma^{n}(f) \}$,
\item $I$ surjects onto $Y^\m \s{n}$ if and only if $c_{n} \in \{ \sigma^{n}(z+ \m ), \sigma^{n}(f) \}$.
\end{enumerate}
\end{lemma}
\begin{proof}This follows from essentially the same proof as Lemma~\ref{cssc}.
\end{proof}

Observe that in the cases of distinct roots or a multiple root, whether a simple module is a factor of a finitely generated graded right $A$-submodule of $\QgrA$ or not depends on only a single structure constant. In the case of congruent roots, each simple factor of the form $Z\s{n}$ is determined by two different structure constants. Additionally, as a consequence of the constructions in Lemmas~\ref{cssc} and \ref{dssc}, we obtain the following two corollaries, the latter of which is an analogue of \cite[Lemma 4.11]{sierra}.

\begin{corollary} \label{homIS} Let $I$ be a finitely generated graded right $A$-submodule of $\QgrA$ and let $S$ be a simple graded $A$-module. Then for each each $n \in \ZZ$
\[ \dim_{\kk} \uHom_A(I, S)_n \leq 1.
\]
\end{corollary}
\begin{proof}We claim that there is a unique kernel $J \subseteq I$ for any surjection $I \sra S$. If $S$ is supported at $\ZZ \cup (\ZZ-\m)$ then this follows from the construction in Lemmas~\ref{cssc} and \ref{dssc}.

If $S = M_\lambda$ for some $\lambda \in \kk$, $\lambda \notin \ZZ \cup \ZZ+\m$, then notice that we can construct $J = (z+\lambda)I$ such that $I/J \cong M_\lambda$. Further, since $M_\lambda$ is simply supported at $-\lambda$ with degree 1 in each graded component, $J$ is the unique kernel for any surjection $I \sra M_\lambda$.

Hence for any simple module $S$, if there is a surjection $I \sra S$, we have
\[ \uHom_A(I, S) \cong \uHom_A(I/J, S) \cong \uHom_A(S,S).
\]
Since all graded simple modules have $\kk$-dimension 0 or 1 in each graded component, we conclude that for each $n \in \ZZ$, $\dim_{\kk} \uHom_A(I, S)_n \leq 1$.
\end{proof}

\begin{corollary} \label{largedegree} Let $I$ be a finitely generated graded right $A$-submodule of $\QgrA$. If $\m \in \NN$, then for all $n \gg 0$, 
\[\Hom_{\gr A}(I, X \s{n}) = \Hom_{\gr A}(I,Y \s{-n}) = \kk.\]

If $\m \in \kk \setminus \ZZ$, then for all $n \gg 0$, 
\[\Hom_{\gr A}(I, X_0 \s{n}) = \Hom_{\gr A}(I,Y_0 \s{-n}) = \kk \mbox { and}\]
\[\Hom_{\gr A}(I, X_\m \s{n}) = \Hom_{\gr A}(I,Y_\m \s{-n})= \kk.\]
\end{corollary}
\begin{proof} The result follows from Lemmas~\ref{cfgsc}, \ref{cssc}, \ref{dssc}, and Corollary~\ref{homIS}.
\end{proof}

\subsection{Finite length modules of $\gr A$}
\label{sec:finlen}
We now seek to understand the finite length modules of $\gr A(f)$. We will see that the extensions between graded simple modules are few: if $M$ and $M'$ are graded simple modules, then $\uExt^1_A(M,M')$ is either $0$ or $\kk$. For clarity, we consider the three cases (congruent, multiple, and non-congruent roots) separately.

\subsubsection{Multiple root}
Let $\m = 0$ so that $f = z^2$ and we are in the multiple root case. We record the $\Ext$ groups between simple modules.

\begin{lemma} \label{mexts}
\begin{enumerate} \item As graded vector spaces, $\uExt^1_A(X,Y) = \uExt^1_A(Y,X) = \kk$, concentrated in degree $0$.
\item As graded vector spaces, $\uExt^1_A(X,X) = \uExt^1_A(Y,Y) = \kk$, concentrated in degree $0$.
\item Let $\lambda, \mu \in \kk \setminus \ZZ$. If $\lambda \neq \mu$, then $\Ext^1_{\gr A}(M_\lambda, M_\mu) = 0$, but $\Ext_{\gr A}^1(M_\lambda, M_\lambda) = \kk$.
\item Let $\lambda \in \kk \setminus \ZZ $. Let $S \in \{X,Y\}$. Then $\uExt^1_A(M_\lambda, S) = \uExt^1_A(S, M_\lambda) = 0$.
\end{enumerate}
\end{lemma}
\begin{proof} Recall that $X = A/(x,z)A$ and $Y = A/(y,z-1)A \s{1}$. Let $I_X$ and $I_Y$ be the ideals of $A$ defining $X$ and $Y$ respectively, that is, $I_X = (x,z)A$ and $I_Y = (y,z-1)A $. Now for $S \in \{X,Y\}$, we have the exact sequence
\begin{equation} 0 \sra I_S \s{d} \sra A\s{d} \sra S \sra 0 \label{mSses}
\end{equation}
where if $S=X$ then $d= 0$ and if $S = Y$ then $d =1$. For any graded simple module $S'$, we then apply the functor $\uHom_A(-,S')$ to yield 
\begin{equation} 0 \sra \uHom_A(S,S') \sra \uHom_A(A, S') \s{-d} \sra \uHom_A(I_S, S') \s{-d} \sra \uExt_A^1(S, S') \sra 0. \label{mSles} 
\end{equation}
We know that $\uHom_A(S,S') = \kk$, concentrated in degree $0$ if and only if $S = S'$, otherwise $\uHom_A(S,S') = 0$. We also know that as a graded $\kk$-vector space, $\uHom_A(A,S')\s{-d} \cong S'\s{-d}$. Additionally, based on Lemma~\ref{cssc} and Corollary~\ref{homIS}, we can compute $\uHom_A(I_S, S') \s{-d}$. Hence, we will able to deduce $\uExt_A^1(S, S')$.

(1) Let $S = X$, $S' = Y$, and $d = 0$ in the exact sequence \eqref{mSles}. Notice that by the construction in Lemma~\ref{cssc}, $A$ and $I_X$ have the same structure constants except in degree $0$, where $A$ has structure constant 1 and $I_X$ has structure constant $z$. Hence, $\uHom_A(A,Y)$ and $\uHom_A(I_X, Y)$ differ only in degree $0$ where $\Hom_{\gr A}(A,Y) = 0$, $\Hom_{\gr A}(I_X, Y) = \kk$ so $\uExt_A^1(X,Y) = \kk$, concentrated in degree $0$. By applying the autoequivalence $\omega$ as in equation \eqref{omegaext}, we deduce $\uExt_A^1(Y,X) = \kk$.

(2) Let $S = S' = X$ and $d = 0$ in the exact sequence \eqref{mSles}. In this case, $\uHom_A(X,X) = \kk$. Again, by the construction in Lemma~\ref{cssc}, $A$ and $I_X$ have the same structure constants except in degree $0$, where $A$ has structure constant 1 and $I_X$ has structure constant $z$. Notice then that $\uHom_A(A,X)_n = \uHom_A(I_X,X)_n$ for every degree $n \in \ZZ$. By considering the exact sequence in each degree, we conclude $\uExt_A^1(X,X) = \kk$, concentrated in degree $0$. Applying $\omega$ yields the result $\uExt_A^1(Y,Y) = \kk$.

(3) Apply $\Hom_{\gr A}(-, M_\mu)$ to the short exact sequence
\begin{equation} \label{mMexact} 0 \sra (z+\lambda)A \sra A \sra M_\lambda \sra 0
\end{equation}
to yield
\begin{align*} 0 &\sra \Hom_{\gr A}(M_\lambda, M_\mu) \sra \Hom_{\gr A}(A, M_\mu) \sra \Hom_{\gr A}((z+\lambda)A, M_\mu) \\
&\sra \Ext_{\gr A}^1(M_\lambda, M_\mu) \sra 0.
\end{align*}
Since $\dim_{\kk} (M_\mu)_0 = 1$, $\Hom_{\gr A}(A, M_\mu) \cong \Hom_{\gr A}((z+\lambda)A, M_\mu) = \kk$. Hence, it follows that $\Ext_{\gr A}^1(M_\lambda, M_\mu) \cong \Hom_{\gr A}(M_\lambda, M_\mu)$ which is $\kk$ when $\lambda = \mu$ and $0$ otherwise.

(4) Let $\lambda \in \kk \setminus \ZZ$ and apply $\uHom_A(-, M_\lambda)$ to the short exact sequence \eqref{mSses}. This yields 
\[ 0 \sra \uHom_A(A, M_\lambda)\s{-d} \sra \uHom_A(I_S, M_\lambda)\s{-d} \sra \uExt_A^1(S, M_\lambda) \sra 0.
\]
Now by Corollary~\ref{homIS}, in every graded component
\[\uHom_A(A, M_\lambda) \cong \uHom_A(I_S, M_\lambda)\]
so $\uExt_A^1(S, M_\lambda) = 0$.

Now let $S \in \{X,Y\}$ and apply $\uHom_A(-,S)$ to the short exact sequence~\eqref{mMexact} yielding
\[ 0 \sra \uHom_A(A,S) \sra \uHom_A((z+\lambda)A, S) \sra \uExt_A^1(M_\lambda, S) \sra 0.
\]
But since $z+\lambda$ has degree $0$, each $\uHom_A((z+\lambda)A, S) \cong \uHom_A(A, S)$ so $\uExt_A^1(M_\lambda, S) = 0$. 
\end{proof}

This lemma allows us to characterize all indecomposable modules of length $2$ in $\gr A$. Since $\uExt_A^1(X,Y) = \kk$, there is a unique (up to isomorphism) nonsplit extension of $X$ by $Y$. We denote this module $E_{X,Y}$ and similarly we denote by $E_{Y,X}$, $E_{X,X}$ and $E_{Y,Y}$ the extensions of $Y$ by $X$, $X$ by $X$, and $Y$ by $Y$, respectively. We record these modules explicitly:
\begin{lemma}\label{mlengthtwo} Let $\m = 0$. The length two indecomposable modules of $\gr A$ whose simple factors are $X$ or $Y$ are precisely the modules
\begin{enumerate}
\item $E_{X,Y} \cong A/zA$;
\item $E_{Y,X} \cong A/(z-1)A \s {1}$;
\item $E_{X,X} \cong A/xA$;
\item $E_{Y,Y} \cong A/yA \s{1}$.
\end{enumerate}
\end{lemma}
\begin{proof}We record the nonsplit exact sequences. To show that $A/zA$ is an extension of $X$ by $Y$, we consider the natural quotient map $A/zA \sra A/(x,z)A = X$ whose kernel is the submodule $(xA + zA)/zA$. We can construct an isomorphism $Y \sra (xA+zA)/zA$ mapping $1$ to $x$. This map is clearly surjective, and is injective since $yA + (z-1)A$ is the right annihilator of $x$ in $A/zA$. Similar arguments apply in the other cases. All four nonsplit exact sequences are recorded below.
\[0 \ra Y \overset{x \cdot}{\ra} A /zA \ra X \ra 0
\]
\[ 0 \ra X \overset{y \cdot}{\ra} (A/(z-1)A) \s{1} \ra Y \ra 0
\]
\[ 0 \ra X \overset{z \cdot}{\ra} A/xA \ra X \ra 0
\]
\[ 0 \ra Y \overset{(z-1)\cdot}{\ra} (A/yA)\s{1} \ra Y \ra 0
\]

These short exact sequences do not split. As an example, in the first case, any nonzero element of non-positive degree generates $A/zA$ so there does not exist a nonzero map $X \sra A/zA$.
\end{proof}

\subsubsection{Congruent roots}
Let $\m \in \NNp$ so that we are in the congruent root case. 

\begin{lemma} \label{cexts}
\begin{enumerate} \item As graded vector spaces, 
\[\uExt^1_A(X,Z) = \uExt^1_A(Z,X) = \uExt^1_A(Y,Z) = \uExt^1_A(Z,Y) = \kk,\] concentrated in degree $0$.
\item Let $S \in \{X,Y,Z\}$. Then $\uExt^1_A(S,S) = 0$.
\item $\uExt^1_A(X,Y) = \uExt^1_A(Y,X) = 0$.
\item If $\lambda \neq \mu$, then $\Ext^1_{\gr A}(M_\lambda, M_\mu) = 0$, but $\Ext_{\gr A}^1(M_\lambda, M_\lambda) = \kk$.
\item Let $\lambda \in \kk \setminus \ZZ $. Let $S \in \{X,Y, Z \}$. Then $\uExt^1_A(M_\lambda, S) = \uExt^1_A(S, M_\lambda) = 0$.
\item \label{extZA} $\uExt^1_A(Z,A) =0$.
\end{enumerate}
\end{lemma}
\begin{proof} We will use similar arguments as in the proof of Lemma~\ref{mexts}. Let $I_X = (x, z+\m )A $, $I_Y = (y, z-1)A $, and $I_Z = (y^\m , x, z)A$ be the right ideals of $A$ defining $X$, $Y$, and $Z$, respectively. Now for $S \in \{X,Y, Z\}$, we have the exact sequence
\begin{equation} 0 \sra I_S\s{d} \sra A\s{d} \sra S \sra 0 \label{cSses}
\end{equation}
where if $S=X$ then $d= -\m $, if $S = Y$ then $d =1$, and if $S = Z$ then $d = 0$. For any graded simple module $S'$, we then apply the functor $\uHom_A(-,S')$ to yield 
\begin{equation} 0 \sra \uHom_A(S,S') \sra \uHom_A(A, S') \s{-d} \sra \uHom_A(I_S, S')\s{-d} \sra \uExt_A^1(S, S') \sra 0. \label{cSles} 
\end{equation}
We will use Lemma~\ref{cssc} and Corollary~\ref{homIS} to compute $\uHom_A(I_S, S') \s{-d}$ from which we will deduce $\uExt_A^1(S, S')$.

(1) Let $S = X$ and $S' = Z$ in the exact sequence \eqref{cSles}. Since $\uHom_A(X,Z) = 0$, we need only determine in which degrees $ \uHom_A(A, Z)\s{\m }$ differs from $\uHom_A(I_X, Z) \s{\m }$. By the construction in Lemma~\ref{cssc}, observe that $A$ and $I_X$ have the same structure constants except in degree $0$ where $A$ has structure constant 1 and $I_X$ has structure constant $z + \m $. Again, by Lemma~\ref{cssc}, $I_X$ surjects onto exactly those shifts of $Z$ that $A$ does except $I_X$ additionally surjects onto $Z\s{\m }$. That is for each $n \in \ZZ \setminus \{0\}$,
\[ \uHom_A(A, Z)\s{\m }_n = \uHom_A(I_X, Z) \s{\m }_n,
\]
In degree $0$, we have that $\Hom_{\gr A}(A, Z)\s{\m } = 0$ and ${\Hom_{\gr A}(I_X, Z) \s{\m } = \kk}$. Hence, $\uExt_A^1(X,Z) = \kk$, concentrated in degree $0$. Applying the autoequivalence $\omega$ implies that $\uExt_A^1(Y,Z) = \kk$.

Now let $S = Z$ and $S' = X$ in the exact sequence \eqref{cSles}. By Lemma~\ref{cssc}, $A$ and $I_Z$ have the same structure constants except in degrees $0$ and $-\m $. In degree $0$, $A$ has structure constant 1 and $I_Z$ has structure constant $z$ whereas in degree $-\m $, $A$ has structure constant $\sigma^{-\m }(f)$ while $I_Z$ has structure constant $\sigma^{-\m }(z )$. Again, by Lemma~\ref{cssc}, $I_Z$ surjects onto exactly those shifts of $X$ that $A$ does except $I_X$ additionally surjects onto $X$. Hence, $\uExt_A^1(Z,X) = \kk$ and by applying the autoequivalence $\omega$, we also see that $\uExt_A^1(Z,Y) = \kk$.

Parts 2 and 3 follow by a similar argument. Parts 4 and 5 follow by the same argument as in Lemma~\ref{mexts}.

(6) Suppose for contradiction that for some $n \in \ZZ$, there exists a nonsplit extension of $Z\s{n}$ by $A$. That is, there exists a graded right $A$-module $I$ such that
\[0 \sra A \sra I \sra Z\s{n} \ra 0
\]
is a nonsplit exact sequence.

We claim that $I$ is isomorphic to a submodule of the graded quotient ring of $A$, so is isomorphic to a right ideal of $A$. We first show that $A \subseteq I$ is an essential extension of modules. Suppose for contradiction that $A \subseteq I$ is not essential so there exists $0 \subsetneq J \subseteq I$ such that $J \cap A = 0$. Observe that we have the inclusion
\[ \frac{J + A}{A} \subseteq \frac{I}{A} \cong Z\s{n}
\]
so $J+A/A$ is a submodule of $Z\s{n}$. Since $Z \s{n}$ is simple, and since $J \cap A = 0$, this implies that $J+A/A \cong Z\s{n}$ so we have the isomorphisms
\[ J \cong \frac{J}{J \cap A} \cong \frac{J+A}{A} \cong \frac{I}{A} \cong Z\s{n}.
\]
Tracing these canonical isomorphisms gives us a splitting $I/A \sra I$, which is a contradiction since $I$ is a nonsplit extension of $A$ by $Z \s{n}$.

Now since $A \sra I$ is an essential extension, we can embed in the injective hull of $A$, which is $\QgrA$, proving our claim. Let $\{d_i\}$ be the structure constants of $A$. By the construction in Lemma~\ref{cssc}, since $A$ is the kernel of the surjection $I \sra Z\s{n}$, therefore $d_{n-\m } \in \{1, \sigma^{n-\m }(z)\}$ and $d_n \in \{\sigma^{n}(z), \sigma^{n}(f)\}$. But by the computation in Example~\ref{scA}, this is impossible. Hence, there are no nonsplit extensions of $Z\s{n}$ by $A$ and $\uExt^1_A(Z,A) =0$. 
\end{proof}

The preceding lemma allows us to characterize all length two indecomposables in $\gr A$. There is a unique (up to isomorphism) nonsplit extension of $X$ by $Z$. We denote this module $E_{X,Z}$ and similarly we denote by $E_{Z,X}$, $E_{Y,Z}$ and $E_{Z,Y}$ the extensions of $Z$ by $X$, $Y$ by $Z$, and $Z$ by $Y$, respectively. We record these modules explicitly:
\begin{lemma}\label{clengthtwo} Let $\m \in \NNp$. The length two indecomposable modules of $\gr A$ whose simple factors are $X$, $Y$, or $Z$ are precisely the modules
\begin{enumerate}
\item $E_{Z,X} \cong A/(x,z)A$;
\item $E_{X,Z} \cong A/(x^{\m +1}, z+\m )A \s {-\m }$;
\item $E_{Z,Y} \cong A/(y, z+\m -1) \s{1-\m }$;
\item $E_{Y,Z} \cong A/(y^{\m +1}, z-1)A \s {1}$.
\end{enumerate}
\end{lemma}
\begin{proof} We will show that $E_{Z,X}$ is a nonsplit extension of $Z$ by $X$. There is a natural projection
\begin{equation} \label{Zsurj}E_{Z,X} = \frac{A}{(x,z)A} \ra \frac{A}{(x,y^{\m }, z)A} = Z
\end{equation}
whose kernel is given by 
\[\frac{(x,y^{\m },z)A}{(x,z)A} \cong \frac{y^{\m }A}{y^{\m }A \cap (x,z)A}.
\]
It is easy to check that there is a well-defined homomorphism
\[ X = \frac{A}{(x,z+\m )A}\s{-\m } \ra \frac{y^{\m }A}{y^{\m }A \cap (x,z)A}
\]
mapping $1$ to $y^{\m }$. Since $X$ is simple and this map is a surjection, therefore the kernel of \eqref{Zsurj} is isomorphic to $X$.

To see that $E_{Z,X}$ is a nonsplit extension of $Z$ by $X$, we will show that $\Hom_{\gr A}(Z, E_{Z,X}) = 0$. If $\varphi$ is a homomorphism $Z \sra E_{Z,X}$, then $\varphi(1) = g$ for some $g \in \kk[z]$. Then $\varphi(y^{\m + 1 }) = gy^{\m + 1 }$. But since $y^{\m + 1 }=0$ in $Z$, therefore $g = 0$ in $E_{Z,X}$ so $\varphi$ is the zero map. The remaining computations are done analogously. By Lemma~\ref{cexts}, these modules represent all length two indecomposable modules of $\gr A$.
\end{proof}

We will also need some understanding of indecomposable modules of length 3. While we will not characterize all such modules, we will show, for various lists of simple modules, that there is a unique indecomposable module with those ordered Jordan-H\"{o}lder quotients. This will allow us to name these modules by specifying their Jordan-H\"{o}lder quotients, in order.

\begin{lemma}Let $\m \in \NNp$.
\begin{enumerate} 

\item $\displaystyle E_{X,Z,X} := \frac{A}{\left(x^{\m +1}, (z+\m )x, (z+\m )^2\right)A} \s{-\m }$ is the unique indecomposable module which is a nonsplit extension of $E_{X,Z}$ by $X$.
\item $\displaystyle E_{Y,Z,Y} := \frac{A}{\left(y^{\m +1}, (z-1)y, (z-1)^2\right)A} \s{1}$ is the unique indecomposable module which is a nonsplit extension of $E_{Y,Z}$ by $Y$.
\item $E_{Z,Y,X} := A/zA$ is the unique indecomposable module which is a nonsplit extension of $E_{Z,Y}$ by $X$ and a nonsplit extension of $E_{Z,X}$ by $Y$.
\end{enumerate}
\end{lemma}
\begin{proof} We will prove the first claim. The other two follow from similar arguments. There is a natural surjection
\[ E_{X,Z,X} = \frac{A}{(x^{\m +1}, (z+\m )x, (z+\m )^2)A} \s{-\m } \ra \frac{A}{(x^{\m + 1},z+\m )A} \s{-\m } = E_{X,Z}
\]
whose kernel is given by
\[ \frac{(x^{\m + 1}, z+\m )A}{(x^{\m +1}, (z+\m)x, (z+\m )^2)A}\s{-\m } \cong \frac{(z+\m )A}{(x^{\m +1}, (z+\m)x, (z+\m )^2)A} \s{-\m }.
\]
This kernel is isomorphic to $X$ via the homomorphism mapping $1$ to $z+ \m $. Therefore, $E_{X,Z,X}$ is an extension of $E_{X,Z}$ by $X$. 

To see that this is a nonsplit extension, suppose $\varphi \in \Hom_{\gr A}(E_{X,Z}, E_{X,Z,X})$. Then $\varphi(1) = g$ for some $g \in \kk[z]$. To ensure that $\varphi$ is well-defined, $\varphi(z+\m ) = 0$, so $z + \m \mid g$. But then $\varphi(x) = gx = 0$ so $\varphi$ is not an isomorphism onto its image.

Finally, we show that $E_{X,Z,X}$ is the unique module that is an extension of $E_{X,Z}$ by $X$. Apply $\uHom_A(-,X)$ to the exact sequence
\[ 0 \sra Z \sra E_{X,Z} \sra X \sra 0
\]
to obtain the sequence
\[ \cdots \sra \uExt_A^1(X,X) \sra \uExt_A^1(E_{X,Z}, X) \sra \uExt_A^1(Z,X) \sra \uExt_A^2(X,X) \sra \cdots.
\]
By \cite[Theorem 5]{bav}, $X$ has projective dimension $1$ so $\uExt_A^2(X,X) = 0$. Then by Lemma~\ref{cexts}, $\uExt_A^1(E_{X,Z},X) = \kk$, concentrated in degree $0$. So there is a unique extension of $E_{X,Z}$ by $X$, namely $E_{X,Z,X}$. 
\end{proof}

\subsubsection{Non-congruent roots}
Let $\m \in \kk \setminus \ZZ$ so that we are in the non-congruent root case. We again begin by careful analysis of the extensions between graded simple $A(f)$-modules.

\begin{lemma} \label{dexts}
\begin{enumerate}
\item As graded vector spaces, $\uExt^1_A(X_0, Y_0) = \uExt^1_A(Y_0, X_0) = \uExt^1_A(X_\m , Y_\m ) = \uExt^1_A(Y_\m , X_\m ) = \kk$, concentrated in degree $0$.
\item Let $S \in \{X_0,Y_0, X_\m , Y_\m \}$. Then $\uExt^1_A(S,S) = 0$.
\item Let $S_0 \in \{X_0, Y_0\}$ and $S_{\m } \in \{X_\m , Y_\m \}$. Then $\uExt^1_A(S_0,S_{\m } )= \uExt^1_A(S_{\m } ,S_0) = 0$.
\item If $\lambda \neq \mu $, then $\Ext^1_{\gr A}(M_\lambda, M_\mu) = 0$, but $\Ext^1_{\gr A}(M_\lambda, M_\lambda) = \kk$.
\item Let $\lambda \in \kk \setminus (\ZZ \cup \ZZ +\m )$. Let $S \in \{X_0,Y_0, X_\m , Y_\m \}$. Then $\uExt^1_A(M_\lambda, S) = \uExt^1_A(S, M_\lambda) = 0$.
\end{enumerate}
\end{lemma}
\begin{proof} We use the same techniques as in our proofs of Lemmas~\ref{mexts} and \ref{cexts}. Let $I_{X_0} = (x,z)A$, $I_{Y_0} = (y,z-1)A$, $I_{X_\m } = (x, z+\m )A$, and $I_{X_\m } = (y, z+\m -1)A$ be the ideals of $A$ defining $X_0$, $Y_0$, $X_\m $ and $Y_\m $ respectively. Now for $S \in \{X_0,Y_0, X_\m , Y_\m \}$, we have the exact sequence
\begin{equation} 0 \sra I_S \s{d} \sra A\s{d} \sra S \sra 0 \label{dSses}
\end{equation}
where if $S=X_0$ or $X_\m $ then $d= 0$ and if $S = Y_0$ or $Y_{ \m }$ then $d =1$. For any graded simple module $S'$, we then apply the functor $\uHom_A(-,S')$ to yield 
\begin{equation} 0 \sra \uHom_A(S,S') \sra \uHom_A(A, S') \s{-d} \sra \uHom_A(I_S, S') \s{-d} \sra \uExt_A^1(S, S') \sra 0. \label{dSles}\end{equation}
Based on Lemma~\ref{dssc} and Corollary~\ref{homIS}, we can compute $\uHom_A(I_S, S') \s{-d}$. Hence, we will able to deduce $\uExt_A^1(S, S')$.

(1) Let $S = X_0$, $S' = Y_0$, and $d = 0$ in the exact sequence \eqref{dSles}. Now $\uHom_A(X_0,Y_0) = 0$ and $\uHom_A(A,Y_0) \cong Y_0$ as graded $\kk$-vector spaces, so to compute $\uExt_A^1(X_0, Y_0)$ we need to compute in which degrees $\uHom_A(I_{X_0}, Y_0)$ is nonzero. By the construction in Lemma~\ref{dssc}, $A$ and $I_{X_0}$ have the same structure constants except in degree $0$, where $A$ has the structure constant $1$ and $I_{X_0}$ has the structure constant $z$. Hence, by Lemma~\ref{dssc} and Corollary~\ref{homIS}, $A$ and $I_{X_0}$ surject onto exactly the same shifts of $Y_0$ except that $I_{X_0}$ surjects onto $Y_0$ while $A$ does not. Therefore, $\uExt_A^1(X_0,Y_0) = \kk$, concentrated in degree $0$. Applying the autoequivalence $\omega$ shows that $\uExt_A^1(Y_{\m },X_{\m }) = \kk$.

An analogous argument shows that $\uExt_A^1(Y_0,X_0) = \uExt_A^1(X_{\m },Y_{\m }) = \kk$, concentrated in degree $0$.

Parts 2 and 3 follow by similar arguments. Parts 4 and 5 follow by the same argument as in Lemma~\ref{mexts}.
\end{proof}

We now know that there is a unique extension of $X_0$ by $Y_0$, $Y_0$ by $X_0$, $X_\m $ by $Y_\m $ and $Y_\m $ by $X_\m $. We classify these extensions explicitly. We leave the proof to the reader, since it is similar to the other cases.

\begin{corollary} \label{dexactsqs} Let $\m \in \kk \setminus \ZZ$. The following nonsplit exact sequences in $\gr A$ represent all length two indecomposables of $\gr A$ whose simple factors are supported at $\ZZ \cup \ZZ - \m $:
\[ \label{dfirstexact}
0 \sra Y_0 \overset{x \cdot}{\ra} A/zA \ra X_0 \ra 0,
\]
\[
0 \ra X_0 \overset{y \cdot}{\ra} (A/(z-1)A) \s{1} \ra Y_0 \ra 0,
\]
\[
0 \ra Y_\m \overset{x \cdot}{\ra} A/(z+ \m )A \ra X_\m \ra 0,
\]
\[
\pushQED{\qed}
0 \ra X_\m \overset{y \cdot}{\ra} (A/(z+ \m -1)A) \s{1} \ra Y_\m \ra 0. \qedhere  \popQED
\]
\end{corollary}

\subsection{Rank one projective modules}
\label{sec:proj}

\subsubsection{Structure constants and projectivity}

We now seek to understand the finitely generated graded right projective $A$-modules of rank one. Since every rank one graded projective $A$-module embeds in $\QgrA$, we may use the results of section~\ref{sec:ideals}. For a finitely generated rank one graded projective $A$-module $P$, if we have a graded embedding $P \subseteq \QgrA$, then there exists some $g \in \kk[z]$, such that $gP \subseteq A$. Since $gP \cong P$, we can view $P$ as a graded right ideal of $A$. We will see that the projectivity of a finitely generated right $A$-submodule of $\QgrA$ can be determined by its structure constants.

By Theorem~\ref{gldimA}, the global dimension of $A(f)$ depends on the roots of $f$. In the non-congruent root case ($\m \in \kk \setminus \ZZ$), $\gldim A(f) = 1$, i.e. $A(f)$ is hereditary. Hence, the isomorphism classes of right ideals of $A$ are the same as the isomorphism classes of rank one projective right $A$-modules. However, in the multiple root case ($\m = 0$) $\gldim A(f) = \infty$ and in the congruent root case ($\m \in \NNp$), $\gldim A(f) = 2$. In these cases, it takes some work to determine whether a finitely generated graded right $A$-submodule of $\QgrA$ is projective. Eventually we will give conditions on structure constants (or equivalently, conditions on simple factors) that determine the projectivity of a right ideal.

In the congruent root case, by \cite[Theorem 5]{bav}, the finite-dimensional simple module $Z$ and its shifts are the only simple modules with projective dimension $2$. We begin our study of the projectivity of a right ideal by seeing that an ideal $P$ is projective if and only if there is no extension of $Z$ by $P$.

\begin{lemma} \label{cprojtfae} Let $\m \in \NNp$ and let $P$ be a right ideal of $A$. Then the following are equivalent:
\begin{enumerate}[(i)]
\item \label{tfae1} $P$ is projective;
\item \label{tfae2} $\uExt_A^1(Z,P) = 0$;
\item \label{tfae3} $\uHom_A(Z, A/P) = 0$;
\item \label{tfae4} $A/P$ has projective dimension at most 1.
\end{enumerate}
\end{lemma}
\begin{proof} The equivalence of \eqref{tfae1} and \eqref{tfae4} is immediate. Now apply the functor $\uHom_A(Z, -)$ to the exact sequence 
\begin{equation} \label{cPsequence} 0 \ra P \ra A \ra A/P \ra 0
\end{equation}
to obtain
\[ 0 \ra \uHom_A(Z, A/P) \ra \uExt^1_A(Z, P) \ra 0,
\]
because by Lemma~\ref{cexts}, $\uExt^1_A(Z,A) = 0$. Hence, \eqref{tfae2} and \eqref{tfae3} are equivalent. 

We now prove that \eqref{tfae1} implies \eqref{tfae3}. Suppose for contradiction that $Z$ is a submodule of $A/P$ for a projective right ideal $P$. Then we have the exact sequence
\[ 0 \ra Z \ra A/P \ra C \ra 0,
\]
where $C$ is the cokernel of the morphism $Z \ra A/P$. Consider any right $A$-module $N$. By applying $\uHom_A(-,N)$ we have the long exact Ext sequence
\[ \cdots \ra \uExt_A^2(A/P,N) \ra \uExt_A^2(Z,N) \ra 0,
\]
since $A$ has global dimension 2. By the equivalence of \eqref{tfae1} and \eqref{tfae4}, $A/P$ has projective dimension 1, so $\uExt_A^2(A/P, N) = 0$ and hence $\uExt_A^2(Z,N)= 0$. Since this holds for all $N$, $Z$ has projective dimension at most 1. But this is a contradiction, for in \cite[Theorem 5]{bav}, Bavula shows that $Z$ has projective dimension 2.

Finally, we show that \eqref{tfae3} implies \eqref{tfae4}, completing the proof. We prove the statement by induction on the length of $A/P$. For simple modules, \cite[Theorem 5]{bav} shows that all simples except $Z$ have projective dimension 1. Now suppose for all modules $M$ of length at most $n$, $\uHom_A(Z,M) = 0$ implies $M$ has projective dimension 1. Suppose $A/P$ has length $n+1$ and $\uHom_A(Z, A/P) = 0$. Then $A/P$ has a submodule $S$ isomorphic to either $X\s{n}$, $Y\s{n}$ or $M_\lambda$ for some $\lambda \in \kk \setminus \ZZ$. The quotient $(A/P)/S$ has length $n$, so by the induction hypothesis, as long as $\uHom_A(Z, (A/P)/S) = 0$, both $S$ and $(A/P)/S$ have projective dimension 1 and thus $A/P$ will have projective dimension 1.

So suppose that $\uHom_A(Z, (A/P)/S) \neq 0$, that is, that $(A/P)/S$ has a submodule isomorphic to $Z \s{n}$ for some $n$. Since this $Z\s{n}$ is not a submodule of $A/P$, in fact $A/P$ must have a submodule that is a nontrivial extension of $Z\s{n}$ by $S$. By Lemma~\ref{clengthtwo}, the only such extensions are $E_{Z,X}\s{n}$ or $E_{Z,Y}\s{n}$. Call this submodule $T$. By \cite[Theorem 5]{bav}, $T$ has projective dimension 1, so as long as $\uHom_A(Z, (A/P)/T) = 0$, we are done by the induction hypothesis.

Suppose for contradiction that $\uHom_A(Z, (A/P)/T) \neq 0$. Arguing as above, $A/P$ must now have a submodule that is a nontrivial extension of $Z\s{m}$ by $T$ for some integer $m$. Now the exact sequence
\[ 0 \ra X \ra E_{Z,X} \ra Z \ra 0
\]
induces the long exact $\Ext$ sequence
\[ 0 \ra \uHom_A(Z,Z) \ra \uExt^1_A(Z, X) \ra \uExt^1_A(Z, E_{Z,X}) \ra 0,
\]
since $\uExt_A^1(Z,Z) = 0$ by Lemma~\ref{cexts}. Again, by Lemma~\ref{cexts}, $\uExt_A^1(Z,X) = \kk$ and since $\uHom_A(Z,Z) = \kk$, we conclude that $\uExt^1_A(Z,E_{Z,X}) = 0$. Hence, there is no nontrivial extension of $Z\s{m }$ by $E_{Z,X}\s{n}$, which is a contradiction. A similar contradiction arises if $T$ is a nontrivial extension of $Z$ by $E_{Z,Y}\s{n}$.
\end{proof}

\begin{corollary}\label{cprojsc} Let $\m \in \NNp$ and let $P$ be a graded right ideal of $A$ with structure constants $\{c_i\}$. Then $P$ is projective if and only if for each $n \in \ZZ$, it is not the case that both $c_{n-\m} \in \{1, \sigma^{n-\m}(z)\}$ and $c_{n} \in\{ \sigma^{n}(z), \sigma^{n}(f)\}$.
\end{corollary}
\begin{proof}Suppose $P = \bigoplus_{i \in \ZZ} (a_i)x^i$ is a projective right ideal of $A$ and suppose for contradiction that for some $n \in \ZZ$, both $c_{n-\m} \in \{1, \sigma^{n-\m}(z)\}$ and $c_{n} \in\{ \sigma^{n}(z), \sigma^{n}(f)\}$. Let 
\[P' = (z+n)P = \bigoplus_{i \in \ZZ} ((z+n)a_i) \kk[z]
\]
and note that $P'$ is graded isomorphic to $P$. We will construct an ideal $I$ of $A$ which is an extension of $Z\s{n}$ by $P'$. 

Let $I = \bigoplus_{i \in \ZZ} (b_i)x^i$ where 
\[ b_i = \begin{cases}(z+n)a_i, & i \leq n-\m \\ a_i, & n-\m+1 \leq i \leq n \\ (z+n)a_i, & i \geq n+1. \end{cases}
\]
This is a well-defined ideal of $A$ as long as $d_i = b_ib_{i+1}^{-1} \in \{ 1, \sigma^{i}(z) , {\sigma^{i}(z+\m)}, \sigma^i(f)\}$ for all $i\in \ZZ$. But for all $i \neq n, n-\m$, by construction $d_i = c_i$. Further, since $c_{n-\m} \in \{1, \sigma^{n-\m}(z)\}$ therefore $d_{n-\m} = (z+n)c_{n-\m} \in \{z+n, \sigma^{n-\m}(f)\}$. And finally $c_{n} \in\{ \sigma^{n}(z), \sigma^{n}(f)\}$ implies that $d_n = (z+n)^{-1}c_n \in \{1, \sigma^{n}(z+\m)\}$.

Now note that $P' \subseteq I$ and $I/P'$ has finite length, is simply supported at $-n$, and is nonzero precisely in degrees between $n-\m+1$ and $n$, inclusive. So $I/P' \cong Z \s{n}$ and $I$ is an extension of $Z\s{n}$ by $P'$. Clearly, such an extension is nontrivial, since $Z\s{n}$ is not a submodule of $A$. But since $P'$ is projective, this contradicts Lemma~\ref{cprojtfae}.

Conversely, suppose $P$ is not projective. By Lemma~\ref{cprojtfae}, there exists a nontrivial extension
\[ 0 \ra P \ra I \ra Z\s{n} \ra 0.
\]
By an argument identical to the one in part \ref{extZA} of Lemma~\ref{cexts}, $I$ is isomorphic to a submodule of $\QgrA$. By clearing denominators we may assume that $I$ is a right ideal of $A$, say $I = \bigoplus_{i \in \ZZ} (b_i)x^i$. But now as in the proof of Lemma~\ref{cssc}, we can construct $P$ as the kernel of the morphism $I \sra Z\s{n}$, and so conclude that $c_{n-\m} \in \{1, \sigma^{n-\m}(z)\}$ and $c_{n} \in\{ \sigma^{n}(z), \sigma^{n}(f)\}$.
\end{proof}

In the multiple root case ($\m = 0$), we get an analogue of Corollary~\ref{cprojsc}, though we use a different technique as there is no analogue of Lemma~\ref{cprojtfae}. Indeed, the following lemma is the same result as if we let $\m = 0$ in Corollary~\ref{cprojsc}.

\begin{lemma}\label{mprojsc}Let $\m = 0 $. Let $P$ be a graded right ideal of $A$ with structure constants $\{c_i\}$. Then $P$ is projective if and only if for each $n \in \ZZ$, $c_n \neq \sigma^n(z)$.
\end{lemma}
\begin{proof}First we show that if, for all $n \in \ZZ$, $c_n \neq \sigma^{n}(z)$, then $P$ is projective. Note that both $xA$ and $yA$ are projective, so $E_{X,X} = A/xA$ and $E_{Y,Y} = A/yA \s{1}$ both have projective dimension $1$. Hence, if $Q$ is any rank one projective that surjects onto $E_{X,X}$ or $E_{Y,Y}$, the kernel of this surjection will also be projective. Since $P$ is finitely generated, by Lemma~\ref{cfgsc}, there exist $N_1, N_2 \in \ZZ$ such that for all $n \geq N_1$, $c_n = 1$ and for all $n \leq N_2$, $c_{n} = \sigma^{n}(f)$. 

We will construct $P$ by constructing a finite sequence of projective modules $\{P_i\}$ until we arrive at a module with the same structure constants as $P$. Let $P_0 = A \s{N_1}$. Clearly $P_0$ is projective. If the structure constants of $P_0$ are $\{ d_i\}$ then since $P_0$ is a shift of $A$, we know that $d_n = 1$ for all $n \geq N_1$ and $d_n = \sigma^n(f)$ for all $n < N_1$. Let $i_0$ be the largest index where $d_{i_0}$ differs from $c_{i_0}$. Since $d_{i_0} = \sigma^{i_0}(f)$, then $c_{i_0} = 1$ (since we are assuming for all $n$ that $c_n \neq \sigma^{n}(z)$). We construct $P_1$ by letting $(P_1)_n = (P_0)_n$ for all $n \leq i_0$ and letting $(P_1)_n = (z+i_0)^2(P_0)_n$ for all $n > i_0$. Then $P_1$ is a submodule of $\QgrA$ with structure constants equal to those of $P_0$ except in degree $i_0$, where $P_1$ has structure constant $1$. Further, the quotient $P_0/P_1$ is supported at $i_0$ and has $\kk$-dimension two in each graded component of degree greater than $i_0$. Hence, $P_0/P_1$ is isomorphic to $E_{Y,Y} \s{i_0}$ or $Y\s{i_0} \oplus Y\s{i_0}$. There is a unique submodule of $P_0$ containing $P_1$ (the module where $P_0$ is multiplied by $z+i_0$ in all degrees greater than $i_0$), and therefore $P_0/P_1 \cong E_{Y,Y}\s{i_0}$. Therefore, $P_1$ is projective.

We continue this process for the finitely many indices where $c_i$ differs from $d_i$ until we reach a module $P_N$ such that $P_0/P_N$ has composition series consisting of only distinct shifts of $E_{Y,Y}$. Since $\pd E_{Y,Y} = 1$, therefore $P_0/P_N$ has projective dimension $1$ and so $P_N$ is projective. Further, since they have the same structure constants, $P_N \cong P$, so $P$ is projective.

Conversely, suppose there exist some indices $n$ such that $c_n = \sigma^n(z)$. Construct a module $Q$ with structure constants $\{d_i\}$ such that $Q$ has the same structure constants as $P$ except if $c_n = \sigma^n(z)$ then $d_n = 1$. By the above argument, $Q$ is projective. Now we construct a finite sequence $\{Q_i\}$ of submodules of $Q$ such that $Q_N \cong P$, and we can show $Q_N$ is not projective. Let $Q_0 = Q$. Let $i_0$ be the largest index where $Q_0$ differs from $P$. Construct $Q_1$ by letting $(Q_1)_n = (Q_0)_n$ for all $n \leq i_0$ and $(Q_1)_n = (z+i_0)(Q_0)_n$ for all $n > i_0$. Then $Q_1$ is a submodule of $\QgrA$ (with structure constants equal to those of $Q_0$ except in degree $i_0$, where $Q_1$ has structure constant $1$). The quotient $Q_0/Q_1$ is supported at $i_0$ and has $\kk$-dimension one in each graded component of degree greater than or equal to $i_0$. Hence, $Q_0/Q_1$ is isomorphic to $Y\s{i_0}$. 

Continue this process for the finitely many indices where $Q$ differs from $P$ to reach a $Q_N$ such that $Q_0/Q_N$ has composition series consisting only of distinct shifts of $Y$. Since there are no extensions between any distinct shifts of $Y$, $Q_0/Q_N$ must be a direct sum of shifts of $Y$, and since $Y$ has infinite projective dimension, so does $Q_0/Q_N$. But since $Q_0$ was projective, this proves that $Q_N$ is not projective. Further, $Q_N \cong P$, proving that $P$ is not projective.
\end{proof}

As a corollary, we see that a graded rank one projective module has a unique simple factor supported at $n$ for each $n \in \ZZ$. As a further corollary, we see that a rank one projective is determined by its integrally supported simple factors.

\begin{corollary}\label{projXYZ} Let $P$ be a rank one graded projective $A$-module. Let $n \in \ZZ$.
\begin{itemize}\item If $\m = 0$, then $P$ surjects onto exactly one of $X\s{n}$ and $Y\s{n}$.
\item If $\m \in \NNp$, then $P$ surjects onto exactly one of $X\s{n}$, $Y\s{n}$, and $Z\s{n}$.
\item If $\m \in \kk \setminus \ZZ$, then $P$ surjects onto exactly one of $X_0\s{n}$ and $Y_0\s{n}$. Likewise, P surjects onto exactly one of $X_{\m}\s{n}$ and $Y_{\m}\s{n}$.
\end{itemize}
\end{corollary}
\begin{proof}If $\m = 0$ or $\m \in \kk \setminus \ZZ$, then this result follows immediately from Corollary~\ref{mprojsc}, Lemma~\ref{dssc}, and Lemma~\ref{cssc}.

If $\m \in \NNp$, then let $\{c_i\}$ be the structure constants of $P$. First, by Lemma~\ref{cssc} if $P$ surjects onto $X\s{n}$ or $Y\s{n}$, then $c_{n} \in \{\sigma^{n}(z), \sigma^n(f)\}$ or $c_{n-\m} \in \{1, \sigma^{n-\m}(z)\}$, so $P$ cannot surject onto $Z\s{n}$. Similarly, if $P$ surjects onto $Z\s{n}$, it cannot surject onto either $X\s{n}$ or $Y\s{n}$. And by Corollary~\ref{cprojsc}, it is not possible that $P$ surjects onto both $X\s{n}$ and $Y\s{n}$. Hence, $P$ surjects onto at most one of $X\s{n}$, $Y\s{n}$, or $Z\s{n}$.

Now we show that $P$ surjects onto exactly one of $X\s{n}$, $Y\s{n}$, or $Z\s{n}$. If $P$ does not surject onto $Y\s{n}$, $c_n \in \{1, \sigma^{n}(z+\m) \}$. If, additionally, $P$ does not surject onto $X\s{n}$ then $c_{n-\m} \in \{\sigma^{n-\m}(z+\m), \sigma^{n-\m}(f)\}$. In either case, by Lemma~\ref{cssc}, $P$ then surjects onto $Z\s{n}$.
\end{proof}

\begin{definition} For a rank one projective $P$ we have shown that for each $j \in \ZZ$, $P$ has a unique simple factor supported at $-j$. We use the notation $F_j(P)$ to refer to this simple factor. Similarly, if $\m \notin \ZZ$, for every $j \in \ZZ$, $P$ has a unique simple factor supported at $-j - \m $. We use the notation $F^{\m }_j(P)$ to refer to this simple factor.
\end{definition}

\begin{corollary}\label{projfactors}A rank one graded projective $A$-module is determined up to isomorphism by its simple factors which are supported at $\ZZ \cup \ZZ-\m$. 
\end{corollary}
\begin{proof} Let $P$ be a rank one graded projective $A$-module with structure constants $\{c_i\}$. In all cases, the integrally supported simple factors determine the $\{c_i\}$ and therefore determine $P$. In the non-congruent root case, the result follows from Lemma~\ref{dssc}. In the multiple root case, this follows from Lemma~\ref{cssc} and Corollary~\ref{projXYZ}. In particular, if $P$ has a factor of $X\s{n}$ then $c_n = 1$ and if $P$ has a factor of $Y\s{n}$ then $c_n = \sigma^{n}(f)$.

In the congruent root case, note that by Corollary~\ref{projXYZ}, for each $n \in \ZZ$, $P$ surjects onto exactly one of $X\s{n}$, $Y\s{n}$ or $Z\s{n}$. Now, for each $j \in \ZZ$, we can determine $c_j$ by using Lemma~\ref{cssc}. In particular, whether or not $X\s{j + \m }$ and $Y\s{j}$ are factors of $P$ completely determine $c_j$. If both $X\s{j + \m }$ and $Y\s{j}$ are factors, then $c_j = \sigma^{j}(z)$. If $X\s{j + \m }$ is a factor but $Y\s{j}$ is not then $c_j = 1$. If $X\s{j + \m }$ is not a factor but $Y\s{j}$ is, then $c_j = \sigma^{j}(f)$. And if neither $X\s{j +\m }$ nor $Y\s{j}$ are factors of $P$, then $c_j = \sigma^j(z+\m )$. We make this explicit in Table~\ref{tablestruc}, below.
\end{proof}

\begin{table}[h!]
\caption{The structure constants of a rank one projective module.}
\label{tablestruc}
\centering
\begin{tabular}{c| ccc}
 & $F_j(P) = X\s{j}$ & $F_j(P) = Y\s{j}$ & $F_j(P) = Z\s{j}$ \\ \hline
$F_{j+\m }(P) = X\s{j + \m }$ & $c_j = 1$ & $c_j = \sigma^{j}(z)$ & $c_j = 1$\\
$F_{j+\m }(P) = Y\s{j + \m }$ & $c_j = \sigma^j(z + \m )$ & $c_j = \sigma^{j}(f)$ & $c_j = \sigma^j(z+\m )$\\
$F_{j+\m }(P) = Z\s{j + \m }$ & $c_j = \sigma^j(z+\m )$ & $c_j = \sigma^{j}(f)$ & $c_j = \sigma^j(z+ \m )$\\
\end{tabular}
\end{table}

\begin{lemma} \label{anyproj}
\begin{enumerate} \item Let $\m \in \NN$. For each $n \in \ZZ$ choose $S_n \in \{X\s{n}, Y\s{n}, Z\s{n}\}$ such that for $n \gg 0$, $S_n = X\s{n}$ and $S_{-n} = Y\s{-n}$. Then there exists a rank one graded projective $P$ such that $F_n(P) = S_n$ for all $n$.

\item Let $\m = 0$. For each $n \in \ZZ$ choose $S_n \in \{X\s{n}, Y\s{n}\}$ such that for $n \gg 0$, $S_n = X\s{n}$ and $S_{-n} = Y\s{-n}$. Then there exists a rank one graded projective $P$ such that $F_n(P) = S_n$ for all $n$.
\end{enumerate}
\end{lemma}
\begin{proof} First, let $\m \in \NN$. We can construct a module $P \subseteq \QgrA$ that surjects onto $S_n$ for all $n$. We do this by specifying the structure constants $\{c_i\}$ of $P$: for each $n$, $c_n$ is determined by $S_n$ and $S_{n+\m }$ as described in Table~\ref{tablestruc}. By the remark following Lemma~\ref{cfgsc}, $P$ is a finitely generated graded right submodule of $\QgrA$.

By Corollary~\ref{cprojsc}, $P$ will be projective as long as, for each $n \in \ZZ$, it is not the case that both $c_{n} \in \{1, \sigma^{n}(z)\}$ and $c_{n+\m } \in\{ \sigma^{n+\m }(z), \sigma^{n+\m }(f)\}$. But if $c_n \in \{1, \sigma^{n}(z)\}$, then by the first row of the table, $P$ surjects onto $X\s{j+ \m }$. If $P$ surjects onto $X\s{j+\m }$, then by the first column of the table, we see that $c_{n+\m} \in \{1, \sigma^{n+\m}(z+\m )\}$. Therefore, in fact $P$ is projective. By construction, $F_n(P) = S_n$ for all $n \in \ZZ$.

The case $\m = 0$ is analogous. In this case, if $S_n = X\s{n}$, then let $c_n = 1$ and if $S_n = Y\s{n}$ then let $c_n = \sigma^{-n}(f)$.
\end{proof}

\subsubsection{Morphisms between rank one projectives}

In this section, we describe the morphisms between finitely generated rank one projective $A$-modules. We again make use of the fact that every right $A$-module is also a left $\kk[z]$-module. The fact that $\kk[z]$ is a PID leads to a nice characterization of the morphisms between rank one projectives. We first note that if $P$ and $Q$ are finitely generated rank one projectives, then any morphism $f \in \Hom_{\gr A}(P,Q)$ is either 0 or else an injection. This is because $A$ is 1-critical and therefore $P/\ker f$ is a finite length module, but $Q$ has no nonzero finite length submodules. Hence, either $\ker f = P$ or else $\ker f = 0$. We define a maximal embedding between right $A$-modules.

\begin{definition} Let $P$ and $Q$ be finitely generated right $A$-modules. An $A$-module homomorphism $f: P \sra Q$ is called a \textit{maximal embedding} if there does not exist an $A$-module homomorphism $g: P \sra Q$ such that $f(P) \subsetneq g(P)$.
\end{definition}

We will prove that in fact, if $P$ and $Q$ are finitely generated rank one projectives, there exists a unique (up to scalar) maximal embedding $P \sra Q$.

\begin{proposition}\label{projhom} Let $P$ and $Q$ be finitely generated graded rank one projective $A$-modules embedded in $\QgrA$. Then every homomorphism $P \sra Q$ is given by left multiplication by some element of $\kk(z)$ and as a left $\kk[z]$-module, $\Hom_{\gr A}(P,Q)$ is free of rank one.
\end{proposition}
\begin{proof}
Since $P$ and $Q$ are finitely generated graded submodules of $\QgrA$, we may multiply by some element of $\kk[z]$ and assume $P$ and $Q$ are right ideals of $A$. Let $f \in \Hom_{\gr A}(P,Q)$. Now since $P_0, Q_0 \subseteq A_0 = \kk[z]$, $P_0$ and $Q_0$ are actually left $\kk[z]$-submodules of $\kk[z]$ and hence ideals of $\kk[z]$. Because $\kk[z]$ is a PID, we can write $P_0 = (p_0)$ and $Q_0 = (q_0)$ for some $p_0, q_0 \in \kk[z]$. Let $f_0$ be the restriction of $f$ to $P_0$. Since $f$ is a right $A$-module homomorphism, it is also a left $\kk[z]$-module homomorphism. Therefore, $f_0$ is given by left multiplication by some $\varphi \in \kk(z)$ such that $(\varphi p_0) \subseteq (q_0)$ in $\kk[z]$.

We claim that $f_0$ determines $f$. Define $m_n = x^n$ for $n \geq 0$ and $m_n = y^n$ for $n < 0$. Now in each graded component, we have that $P_n = (p_n)m_n$ and $Q_n =(q_n)m_n$ where $p_n, q_n \in \kk[z]$. Let $f_n$ be the restriction of $f$ to $P_n$. Note that $(p_0) m_n$ is a nonzero submodule of $(p_n)m_n$, and
\[ f(p_0 m_n) = f(p_0)m_n = \varphi p_0 m_n,
\]
so on the submodule $(p_0)m_n \subseteq P_n$, $f$ is given by left multiplication by $\varphi$. Viewing $f_n$ as a left $\kk[z]$-module homomorphism, we conclude that $f_n$ is given by left multiplication by $\varphi$ for each $n$, so $f$ is simply left multiplication by $\varphi$. Conversely, left multiplication by $\varphi \in \kk(z)$ will be an element of $\Hom_{\gr A}(P,Q)$ if and only if $(\varphi p_n) \subseteq (q_n)$ for all $n$.

Now $\Hom_{\gr A}(P,Q)$ is a left $\kk[z]$-submodule of $\Hom_{\gr \kk[z]}(P,Q)$ which the above shows is isomorphic to some left $\kk[z]$-submodule of $\kk(z)$. Since in addition multiplication by $\varphi$ is a homomorphism only if $\varphi \in \kk[z]q_0/p_0$, we can clear denominators so that $(p_0)\Hom_{\gr A}(P,Q)$ is actually a left $\kk[z]$-submodule of $\kk[z]$. Since $\kk[z]$ is a PID, $(p_0) \Hom_{\gr A}(P,Q) = (g)$ for some $g \in \kk[z]$. Finally, we see that $\Hom_{\gr A}(P,Q)$ is generated as a left $\kk[z]$-module by $g/p_0 \in \kk(z)$. Let $\theta = g/p_0$. Then all homomorphisms $\Hom_{\gr A}(P,Q)$ are given by left multiplication by some $\kk[z]$-multiple of $\theta$. 
\end{proof}

The above proposition tells us that a $\Hom$ set between rank one projective $A$-modules is generated as a left $\kk[z]$-module by a unique (up to scalar) maximal embedding. Note, however, the element $\theta \in \kk(z)$ as in the previous lemma depends on a fixed embedding of $P$ and $Q$ in $A$. As an example, $zA \cong A$, but $\Hom_{\gr A}(zA,A)$ is generated as $\kk[z]$-module by left multiplication by $z^{-1}$ while $\Hom_{\gr A}(A,A)$ is generated by left multiplication by $1$. 

Given rank one projectives $P$ and $Q$ we would like to have canonical representations of $P$ and $Q$ as submodules of the graded quotient ring $\QgrA$. This would also give us a canonical representation of the maximal embedding $P \sra Q$. We will do this graded component by graded component. Suppose $P$ has structure constants $\{c_i\}$. We will give a canonical sequence $\{p_i\}$ such that 
\[P \cong \bigoplus_{i \in \ZZ} (p_i) x^i \subseteq Q_{\mathrm{gr}}(A).
\] By Lemma~\ref{cfgsc}, there exists an integer $N$ such that for all $n > N$, $c_n = 1$. Now recalling that $p_i = c_i p_{i+1}$, for any $n \in \ZZ$, let $p_n = \prod_{j \geq n} c_j$. Observe that $p_n \in \kk[z]$ since all but finitely many of the factors are $1$. 

\begin{definition} Given a graded rank one projective module $P$, the representation of $\bigoplus (p_i) x^i \subseteq \QgrA$ given above is called \emph{canonical representation of $P$}. We call $\bigoplus (p_i) x^i$ a \emph{canonical rank one projective module}.
\end{definition}
Now given canonical graded rank one projective modules $P = \bigoplus (p_i) x^i$ and $Q = \bigoplus (q_i)x^i$, we can compute the maximal embedding of $P$ into $Q$, which is unique if we require that $\theta$ be monic. 

\begin{lemma} \label{projmaxemb} Let $P$ and $Q$ be rank one graded projective $A$-modules with structure constants $\{c_i\}$ and $\{d_i\}$, respectively. As above, write 
\[ P = \bigoplus_{i \in \ZZ} (p_i) x^i = \bigoplus_{i \in \ZZ} \left( \prod_{j \geq i} c_j \right) x^i \mbox{ and } Q = \bigoplus_{i \in \ZZ}(q_i)x^i = \bigoplus_{i \in \ZZ} \left( \prod_{j \geq i} d_j \right) x^i.
\]
Then the maximal embedding $P \sra Q$ is given by multiplication by 
\[
\theta_{P,Q} = \lcm_{i \in \ZZ} \left( \frac{q_i}{\gcd(p_i, q_i)}\right)= \lcm_{i \in \ZZ} \left( \frac{\prod_{j \geq i}d_j}{\gcd\left( \prod_{j \geq i}c_j, \prod_{j \geq i} d_j\right)} \right)
\]
where the $\lcm$ is the unique monic least common multiple $\theta_{P,Q} \in \kk[z]$.
\end{lemma}
\begin{proof} As we saw in Proposition~\ref{projhom}, the maximal embedding $P \sra Q$ is given by multiplication by some $\theta_{P,Q} \in \kk(z)$ such that $(\theta_{P,Q} p_i) \subseteq (q_i)$ for all $i \in \ZZ$. Because, for large enough $i$, $p_i = q_i = 1$, we see that $\theta_{P,Q} \in \kk[z]$. In fact, that $(\theta_{P,Q} p_i) \subseteq (q_i)$ implies that $\theta_{P,Q}$ must be a $\kk[z]$-multiple of $q_i/\gcd(p_i, q_i)$ for all $i$. The minimal such $\theta_{P,Q}$ is given by
\[\theta_{P,Q} = \lcm_{i \in \ZZ} \left( \frac{q_i}{\gcd(p_i, q_i)}\right)\]
Since the structure constants of a graded submodule of $\QgrA$ were defined to be monic, $\theta_{P,Q}$ is monic. Because there exists some $N \in \NN$ such that for all $n \geq N$, $c_n = d_n$ and $c_{-n} = d_{-n}$, therefore $q_i/\gcd(p_i,q_i) = q_j/\gcd(p_j,q_j)$ for all $i,j \geq N$ and all $i,j \leq -N$, so the least common multiple can be taken over finitely many indices. 
\end{proof}

\begin{corollary} \label{dmaxemb} Suppose $\m \in \kk \setminus \ZZ$ or $\m = 0$. Assume the hypotheses of Lemma~\ref{projmaxemb}. Then there exists an $N \in \ZZ$ such that
\[\theta_{P,Q} = \frac{q_N}{\gcd(p_N, q_N)}= \frac{\prod_{j \geq N}d_j}{\gcd\left( \prod_{j \geq N}c_j, \prod_{j \geq N} d_j\right)} = \frac{\prod_{j \geq N}d_j}{\prod_{j \geq N}\gcd\left( c_j, d_j\right)}.
\]
\end{corollary}
\begin{proof}
If $\m \in \kk \setminus \ZZ$ or $\m = 0$, we can compute the maximal embedding $P \sra Q$ without taking a least common multiple, as follows. The irreducible factors of $\theta_{P,Q}$ are all $\sigma^{i}(z)$ or $\sigma^{i}(z+\m )$ for some $i \in \ZZ$. If $\m \in \kk \setminus \ZZ$ or $\m = 0$, then $\sigma^{i}(z)$ and $\sigma^{i}(z+\m )$ can only appear as factors of the structure constant in degree $i$. Thus, if $i \neq j$, then $q_i$ shares no irreducible factors with $p_j$ or $q_j$. Hence, if $n < m$, then $q_m/\gcd (p_m,q_m)$ divides $q_n /\gcd (p_n,q_n)$. By Lemma~\ref{cfgsc}, there exists an $N \in \ZZ$ such that for all $i \leq N$, $p_i = q_i$. For this $N$, we can write
\[\theta_{P,Q} = \frac{q_N}{\gcd(p_N, q_N)}= \frac{\prod_{j \geq N}d_j}{\gcd\left( \prod_{j \geq N}c_j, \prod_{j \geq N} d_j\right)} = \frac{\prod_{j \geq N}d_j}{\prod_{j \geq N}\gcd\left( c_j, d_j\right)}. \qedhere
\] 
\end{proof}

We now show that the cokernel of a maximal embedding has a special structure.

\begin{lemma} \label{maxembfactor} Let $P$ and $Q$ be graded rank one projective $A$-modules and let $f: P \sra Q$ be a maximal embedding. Then the module $Q/f(P)$ is supported at $\ZZ \cup \ZZ-\m $.
\end{lemma}
\begin{proof} Let $N = Q/f(P)$. As $A$ has Krull dimension 1, $N$ has finite length. By Lemma~\ref{simples}, $N$ has a finite composition series whose factors are either supported at $\ZZ \cup \ZZ-\m $ or isomorphic to some $M_\lambda$ with $\lambda \not \in \ZZ \cup \ZZ-\m $. Suppose for contradiction that $M_\lambda$ is a subfactor of $N$. Then we have $f(P) \subseteq Q_1 \subseteq Q_2 \subseteq Q$ with $Q_2/Q_1 \cong M_\lambda$ and so $Q_1 = (z+\lambda)Q_2$. But now $(z+\lambda)^{-1}f(P) \subseteq (z+\lambda)^{-1}Q_1 = Q_2 \subseteq Q$, contradicting the maximality of $f$.
\end{proof}

Finally, we give necessary and sufficient conditions for a collection of projective objects in $\gr A$ to generate the category.
\begin{proposition}\label{projgen} A set of rank one graded projective $A$-modules $\PPPP = \{P_i\}_{i \in I}$ generates $\gr A$ if and only if for every graded simple module $M$ which is supported at $\ZZ \cup \ZZ-\m$, there exists a surjection to $M$ from a direct sum of modules in $\PPPP$.
\end{proposition}
\begin{proof}One direction is clear. Now suppose $\PPPP$ is a set of graded projective $A$-modules that generates all graded simple modules supported at $\ZZ \cup \ZZ-\m$. We will show that every shift of $A$ is the image of a surjection from a direct sum of modules in $\PPPP$, and so $\PPPP$ generates $\gr A$. Let $P \in \PPPP$ and choose a maximal embedding $\varphi: P \sra A \s{n}$. It suffices to construct a surjection ${\psi: \bigoplus_{j \in J} P_j \sra A\s{n}/P}$ for some $J\subseteq I$. This is because, by the projectivity of the $P_j$, there exists a lift $\overline{\psi}: \bigoplus_{j \in J} P_j \sra A\s{n}$ and because $\im \varphi + \im \overline{\psi} = A\s{n}$.

Since $A$ has Krull dimension 1, the quotient $A\s{n}/P$ has finite length. By Lemma~\ref{maxembfactor}, it is supported at $\ZZ \cup \ZZ-\m$. We induct on the length of $A\s{n}/P$. Now there exists some integrally supported simple module $M_0$ which fits into the exact sequence
\[ 0 \ra K_0 \ra A/P \ra M_0 \ra 0.
\]
Again, it suffices to give surjections onto $M_0$ and $K_0$. By hypothesis, $\PPPP$ generates $M_0$. By induction, $\PPPP$ generates $K_0$, completing the proof.
\end{proof}

\section{Functors defined on subcategories of projectives}
\label{sec:func}

The results of section~\ref{sec:proj} suggest that much information about the category $\gr A$ is contained in its rank one projective modules and the morphisms between them. We will therefore develop the machinery necessary to define a functor on $\gr A$ by first defining it on the full subcategory of direct sums of rank one projective modules. We remark that a dual statement to Lemma~\ref{projfunc} is stated in \cite[Proposition 3.1.1(3)]{vdb} (in terms of injective objects), but no detailed proof is given.

\begin{lemma}\label{projfunc} Let $\CCC$ be an abelian category with enough projectives. Let $\PPP$ be a full subcategory consisting of projective objects in $\CCC$ such that every object in $\CCC$ has a projective resolution by objects in $\PPP$. Given an additive functor $\FFF: \PPP \sra \CCC$, there is a unique (up to natural isomorphism of functors) extension of $\FFF$ to a functor $\tilde{\FFF}: \CCC \sra \CCC$ which is right exact. \pushQED{\qed} \qedhere \popQED
\end{lemma}
The above lemma follows from standard homological arguments, and a detailed proof is given in the author's PhD thesis \cite{won}. We briefly review the construction of $\tilde{\FFF}$. For each $M\in \CCC$, fix a partial projective resolution
\[ P_1 \overset{d_0}{\ra} P_0 \ra M \ra 0
\]
where $P_0, P_1$ are objects in $\PPP$. Then apply $\FFF$ to yield
\[ \FFF(P_1) \overset{\FFF(d_0)}{\ra} \FFF(P_0) 
\]
and let $\tilde{\FFF}(M) = \coker \FFF(d_0)$ so there is a right exact sequence
\[ \FFF(P_1) \overset{\FFF(d_0)}{\ra} \FFF(P_0) \ra \tilde{\FFF}(M) \ra 0.
\]
For $P \in \PPP$, we always fix the resolution
\[  0 \ra P \ra P \ra 0
\]
so that $\tilde{\FFF}(P) = \FFF(P)$. 

Next we define $\tilde{\FFF}$ on morphisms. Let $f: M \sra N$ be a morphism and let $P_1 \sra P_0 \sra M$ and $Q_1 \sra Q_0 \sra N$ be the partial projective resolutions of $M$ and $N$, respectively. Because $P_0$ and $P_1$ are projective, there exist lifts of $f$, homomorphisms $h_0$ and $h_1$ such that
\begin{align}\label{chdiagram} \xymatrix{ 
P_1 \ar[d]_{h_1} \ar[r] & P_0 \ar[d]_{h_0} \ar[r] & M \ar[d]_{f} \ar[r] & 0 \\
Q_1 \ar[r] & Q_0 \ar[r] & N \ar[r] & 0 
}
\end{align}
commutes. Although $h_0$ and $h_1$ are not necessarily unique, they do induce unique maps on homology (i.e. any choices for $h_0, h_1$ are chain homotopic). Then applying $\FFF$ to this commutative diagram yields
\begin{align*}\xymatrix{ 
\FFF(P_1) \ar[d]_{\FFF(h_1)} \ar[r] & \FFF(P_0) \ar[d]_{\FFF(h_0)} \ar[r] & \tilde{\FFF}(M) \ar[r] & 0 \\
\FFF(Q_1) \ar[r] & \FFF(Q_0) \ar[r] & \tilde{\FFF}(N) \ar[r] & 0 
}
\end{align*}
which induces a unique map $\tilde{\FFF}(f): \tilde{\FFF}(M) \sra \tilde{\FFF}(N)$ such that
\begin{align*}\xymatrix{ 
\FFF(P_1) \ar[d]_{\FFF(h_1)} \ar[r] & \FFF(P_0) \ar[d]_{\FFF(h_0)} \ar[r] & \tilde{\FFF}(M) \ar[d]_{\tilde{\FFF}(f)} \ar[r] & 0 \\
\FFF(Q_1) \ar[r] & \FFF(Q_0) \ar[r] & \tilde{\FFF}(N) \ar[r] & 0 
}
\end{align*}
commutes. Further, since homotopic maps stay homotopic after applying a functor, and since any choices of $h_0, h_1$ were homotopic, therefore $\FFF(h_0)$ and $\FFF(h_1)$ are also unique up to homotopy. Hence, the induced map on zeroth homology, $\tilde{\FFF}(f)$ is well-defined.

As a corollary of the above lemma, we prove that we can check if two categories are equivalent by checking if they are equivalent on a full subcategory of projective objects.

\begin{corollary} \label{subequiv} Let $\CCC$ and $\CCC'$ be abelian categories with enough projectives. Let $\PPP$ (respectively $\PPP'$) be a full subcategory of projective objects such that every object of $\CCC$ (respectively $\CCC'$) has a projective resolution by objects of $\PPP$ (respectively $\PPP'$). Then if $\FFF: \PPP \sra \PPP'$ is an equivalence of categories, then there exists an equivalence of categories $\widetilde{\FFF}:\CCC \sra \CCC'$ which extends $\FFF$.
\end{corollary}
\begin{proof} Suppose $\FFF: \PPP \sra \PPP'$ is an equivalence of categories with quasi-inverse $\GGG: \PPP' \sra \PPP$. We may regard $\FFF$ as a functor $\FFF: \PPP \sra \CCC'$ and $\GGG$ as a functor $\GGG: \PPP' \sra \CCC$. Now by Lemma~\ref{projfunc}, there exist functors $\widetilde{\FFF}: \CCC \sra \CCC'$ and $\widetilde{\GGG}: \CCC' \sra \CCC$ which extend the functors $\FFF$ and $\GGG$.

Now $\GGG\circ \FFF: \PPP \sra \PPP$ is naturally isomorphic to the identity functor $\Id_{\PPP}$. Consider the composition $\widetilde{\GGG} \circ \widetilde{\FFF}:\CCC \sra \CCC$. Then $\widetilde{\GGG} \circ \widetilde{\FFF}$ is an extension of $\GGG \circ \FFF$ to the category $\CCC$ so by Lemma~\ref{projfunc} is the unique such extension up to natural isomorphism. Since $\Id_{\CCC}$ is an extension of $\Id_{\PPP}$, we conclude that $\widetilde{\GGG} \circ \widetilde{\FFF} \cong \Id_{\CCC}$. Similarly, $\widetilde{\GGG} \circ \widetilde{\FFF} \cong \Id_{\CCC'}$ and hence, $\CCC \equiv \CCC'$.
\end{proof}

We have shown that in order to define a functor on a category, it suffices to construct a functor on ``enough" of the projective objects in that category. In the next section, we will construct autoequivalences of $\gr A$ in this way. We now show that, in fact, the subcategory consisting of only canonical rank one projective modules is big enough, as additive functors defined on these objects extend uniquely to direct sums.

\begin{lemma} \label{projfunc2}Let $\CCC$ be an abelian category. Let $\RRR$ be the full subcategory of $\gr A$ consisting of the canonical rank one projective modules and let $\PPP$ be the full subcategory of $\gr A$ consisting of all finite direct sums of canonical rank one projective modules. If $\FFF: \RRR \sra \CCC$ is an additive functor, then $\FFF$ extends to an additive functor $\widetilde{\FFF}: \PPP \sra \CCC$.
\end{lemma}
\begin{proof}We begin by defining $\widetilde{\FFF}$ on objects. For $P \in \PPP$, choose canonical rank one projective modules $P_i$ and write $P = \bigoplus_{i = 1}^n P_i$. Define $\widetilde{\FFF}(P) = \bigoplus_{i=1}^n \FFF (P_i)$.

We now define $\widetilde{\FFF}$ on morphisms. Let $P, Q, R \in \PPP$ and write $P = \bigoplus_{i = 1}^n P_i$, $Q = \bigoplus_{j = 1}^m Q_j$, and $R = \bigoplus_{k = 1}^r R_k$. First we observe that
\[ \Hom_{\gr A}(P,Q) = \Hom_{\gr A} \left( \bigoplus_{i = 1}^n P_i, \bigoplus_{j = 1}^m Q_j\right) \cong \bigoplus_{i = 1}^n  \bigoplus_{j = 1}^m \Hom_{\gr A} \left( P_i, Q_j\right)
\]
so we can represent $\varphi \in \Hom_{\gr A}(P,Q)$ as a sum $\varphi = \sum_{i=1}^n \sum_{j =1}^m \varphi_{i,j}$ where $\varphi_{i,j} \in \Hom_{\gr A}(P_i, Q_j)$. Define 
\[ \widetilde{\FFF}(\varphi) = \sum_{i =1}^n \sum_{j=1}^m \FFF(\varphi_{i,j}).
\]

Now we can write $\Id_P = \sum_{i=1}^n \Id_{P_i}$ and use the fact that $\FFF$ is a functor to deduce that 
\[\widetilde{\FFF}\left(\Id_P\right) = \sum_{i=1}^n \FFF\left(\Id_{P_i}\right) = \sum_{i=1}^n \Id_{\FFF(P_i)} = \Id_{\widetilde{\FFF}(P)}.
\]
We now check that $\widetilde{\FFF}$ preserves compositions. If $\psi \in \Hom_{\gr A}(Q,R)$, then we can write
\[ \psi \circ \varphi = \sum_{i=1}^n \sum_{k=1}^r (\psi \circ \varphi)_{i,k} = \sum_{i=1}^n \sum_{j=1}^m \sum_{k=1}^r \psi_{j,k} \circ \varphi_{i,j}
\]
so, since $\FFF$ is an additive functor,
\begin{align*} \widetilde{\FFF}(\psi \circ \varphi) &= \sum_{i=1}^n \sum_{k=1}^r \FFF\left((\psi \circ \varphi)_{i,k}\right) = \sum_{i=1}^n \sum_{j=1}^m \sum_{k=1}^r \FFF\left( \psi_{j,k} \circ \varphi_{i,j} \right) \\
&= \sum_{i=1}^n \sum_{j=1}^m \sum_{k=1}^r \FFF(\psi_{j,k}) \circ \FFF(\varphi_{i,j}) \\
&= \sum_{j=1}^m \sum_{k=1}^r \FFF( \psi_{j,k}) \circ \sum_{i=1}^n \sum_{j=1}^m \FFF(\varphi_{i,j}) = \widetilde{\FFF}(\psi) \circ \widetilde{\FFF}(\varphi).
\end{align*}
Hence, $\widetilde{\FFF}$ preserves compositions of morphisms and so $\widetilde{\FFF}: \PPP \sra \CCC$ is a functor extending $\FFF$.
\end{proof}

The remaining results in this section give the technical tools used to construct autoequivalences of $\gr A$ switching $X$ and $Y$.

\begin{lemma} \label{projfunc3} Let $\CCC$ be an abelian category with enough projectives, let $\PPP$ be a full subcategory of $\CCC$ consisting of projective objects, and let $\DDD \subseteq \CCC$ be a full subcategory which is closed under subobjects. Suppose that for every $M \in \PPP$, there exists a unique smallest subobject $N \subseteq M$ such that $M/N \in \DDD$. Write $N = \FFF(M)$. Then there is an additive functor $\FFF: \PPP \sra \CCC$ where for each projective $P$, $\FFF(P)$ is as defined above, with the action of $\FFF$ on morphisms being restriction.
\end{lemma}
\begin{proof} Let $P$ and $Q$ be projective objects and $f \in \Hom_{\CCC}(P,Q)$. If we can show that $f(\FFF(P)) \subseteq \FFF(Q)$ then $f|_{\FFF(P)}: \FFF(P) \sra \FFF(Q)$ is a well-defined restriction, and it is easy to see the functor $\FFF: \PPP \sra \CCC$ defined as above is additive.

Now $Q/\FFF(Q) \in \DDD$ by the definition of $\FFF(Q)$ and $P/f^{-1}(\FFF(Q))$ embeds into $Q/\FFF(Q)$. Since $\DDD$ is closed under subobjects, $P/f^{-1}(\FFF(Q)) \in \DDD$. Then $\FFF(P) \subseteq f^{-1}(\FFF(Q))$ because $\FFF(P)$ was defined to be the unique smallest $N \subseteq P$ such that $P/N \in \DDD$. Hence, $f(\FFF(P)) \subseteq \FFF(Q)$, as desired.
\end{proof}

The preceding lemma is the main tool we will use to construct our autoequivalence. We will construct a full subcategory $\DDD$ of $\gr A$ and a functor $\iota_0$ that maps a rank one projective module to the smallest kernel of morphisms to elements of $\DDD$.

\begin{proposition}\label{projfunc4} Let $\CCC$ be an abelian $\kk$-linear category. Let $\III = \{I_1, \ldots, I_n \}$ be a finite set of indecomposable objects in $\CCC$ and let $\DDD \subseteq \CCC$ be the full subcategory consisting of all finite direct sums of elements of $\III$, possibly with repeats. Suppose further for every $C \in \CCC$ and every $D \in \DDD$ that $\Hom_{\CCC}(C,D)$ is finite-dimensional and that $\DDD$ is closed under subobjects. Then for each $M \in \CCC$, there exists a unique smallest $N \subseteq M$ such that $M/N \in \DDD$.
\end{proposition}
\begin{proof}
Let $M$ be an object in $\CCC$. By hypothesis, for each $1 \leq i \leq n$, $\Hom_{\CCC}(M, I_i)$ is finite-dimensional, say spanned by $\varphi_{i_1}, \ldots, \varphi_{i_d}$. Note that for $1 \leq j \leq d$, $M/\ker \varphi_{i_j} \subseteq I_i$ is an object in $\DDD$, because $\DDD$ is closed under subobjects.

Define $J_i = \cap_{j=1}^d \ker \varphi_{i_j}$. First, since $M/J_i \subseteq \oplus_{j=1}^d M/ \ker \varphi_{i_j}$ and $\DDD$ is closed under direct sums and subobjects, $M/J_i \in \DDD$. Further, for any $\psi \in \Hom_{\CCC}(M,I_i)$, $J_i \subseteq \ker \psi$. 

Now consider the intersection $N = \cap_{i=1}^n J_i \subseteq M$. Again, since each $M/J_i \in \DDD$, we have $M/N \in \DDD$. To show that $N$ is the unique smallest such object, let $L = \oplus_{j=1}^r I_{\alpha_j}$ be an object in $\DDD$ and let $\psi \in \Hom_{\CCC}(M, L)$. Because $\CCC$ is an abelian category, $ \Hom_{\CCC}(M, L) = \Hom_{\CCC}(M, \oplus_{j=1}^r I_{\alpha_j}) = \oplus_{j=1}^r \Hom_{\CCC}(M, I_{\alpha_j})$. But as $N = \cap_{i=1}^n \cap_{j=1}^d \ker \varphi_{i_j}$, therefore $N \subseteq \ker \psi$.
\end{proof}

\section{The Picard group of $\gr A$}
\label{sec:pic}

In this section, we determine the Picard group of $\gr A$. Let $\ZZ_{\fin}$ be the group of finite subsets of $\ZZ$ with operation $\oplus$ given by exclusive or. For the first Weyl algebra, $A_1$, in \cite[Corollary 5.11]{sierra} Sierra computed that 
\[ \Pic(\gr A_1) \cong \ZZ_\fin \rtimes D_\infty.
\] 
The subgroup of $\Pic(\gr A_1)$ isomorphic to $D_\infty$ is generated by the shift functor $\SSS_{A_1}$ and the autoequivalence reversing the graded structure of $A_1$. For quadratic polynomials $f$, $\gr A(f)$ is still equipped with both a shift functor $\SSS_{A(f)}$ as well as the grading-reversing autoequivalence $\omega$. We therefore expect $D_\infty$ to appear as a subgroup of $\Pic(\gr A(f))$. 

The subgroup of $\Pic(\gr A_1)$ isomorphic to $\ZZ_\fin$ is generated by autoequivalences that Sierra calls \emph{involutions} of $\gr A_1$. In section~\ref{sec:iotas}, we will construct analogous involutions of $\gr A(f)$, and so we will show that for any quadratic polynomial $f \in \kk[z]$, 
\[\Pic(\gr A(f)) \cong \ZZ_\fin \rtimes D_\infty \cong \Pic(\gr A_1).
\]

\subsection{The rigidity of $\gr A$} \label{rigid}

In this section, we will show that $\gr A$ exhibits the same sort of rigidity as $\gr A_1$. The general structure of our arguments parallel \cite[\S 5]{sierra}. In particular, we will first prove that an autoequivalence of $\gr A$ is determined by its action on only the simple modules supported at $\ZZ \cup \ZZ -\m $. We will then prove that any autoequivalence of $\gr A$ must permute the simple modules in a rigid way.

We begin by proving an analogue of \cite[Lemma 5.1]{sierra}. Recall from section~\ref{sec:graded}, that for a $\ZZ$-graded ring $R$, we defined the $\ZZ$-algebra associated to $R$
\[ \ol{R} = \bigoplus_{i,j \in \ZZ} \ol{R}_{i,j}
\]
where $\ol{R}_{i,j} = R_{j-i}$. Recall also that an automorphism $\gamma$ of $\ol{R}$ is called \emph{inner} if for all $m,n \in \ZZ$, there exist $g_m \in \ol{R}_{m,m}$ and $h_n \in \ol{R}_{n,n}$ such that for all $w \in \ol{R}_{m,n}$, $\gamma(w) = g_m w h_n$.

We will first study the automorphisms of $\ol{A}$. Following the notation of Sierra, define $m_{ij} \in \ol{A}_{ij} = A_{j-i}$ to be the canonical $\kk[z]$-module generator of $A_{j-i}$; that is, $m_{ij}$ is $x^{j-i}$ if $j \geq i$ and $y^{i-j}$ if $i > j$.
\begin{lemma}\label{inneraut}
Every automorphism of $\ol{A}$ of degree $0$ is inner.
\end{lemma}
\begin{proof} Let $\gamma$ be an automorphism of $\ol{A}$. Since $\gamma$ is an automorphism, for all $i,j \in \ZZ$ there is a unit $\zeta_{ij} \in \kk[z]$ such that $\gamma(m_{ij}) = \zeta_{ij}m_{ij}$. Thus, $\zeta_{ij} \in \kk^*$. Further, for all $n \in \ZZ$, we have $\zeta_{nn}=1$. Denote by $\gamma_n$ the restriction of $\gamma$ to $\ol{A}_{nn} = \kk[z]$. Since $\gamma_n$ is an automorphism of $\kk[z]$, we know it is of the form $z \mapsto a_nz + b_n$ for some $a_n,b_n \in \kk$.

Applying $\gamma$ to the identity
\[ m_{n,n+1}m_{n+1,n} = z(z+\m ) \cdot 1_n
\]
yields
\[ \zeta_{n,n+1}\zeta_{n+1,n} m_{n,n+1}m_{n+1,n} = [a_n^2 z^2 + a_n(2b_n + \m )z + b_n(b_n+\m )] \cdot1_n.
\]
Since, in addition
\[ \zeta_{n,n+1}\zeta_{n+1,n} m_{n,n+1}m_{n+1,n} = (\zeta_{n,n+1}\zeta_{n+1,n}z^2 + \zeta_{n,n+1}\zeta_{n+1,n}\m z) \cdot1_n,
\]
by comparing coefficients, either $b_n = 0$ or $b_n = -\m $. If $b_n = 0$, then $a_n = \zeta_{n,n+1}\zeta_{n+1,n} = a_n^2$. Since $\zeta_{n,n+1}\zeta_{n+1,n}$ is a unit, therefore $a_n = \zeta_{n,n+1}\zeta_{n+1,n} = 1$. On the other hand, if $b_n = -\m $, then $-a_n = \zeta_{n,n+1}\zeta_{n+1,n} = a_n^2$. In this case, $a_n = -1$. In either case, $\zeta_{n,n+1}\zeta_{n+1,n} = 1$.

Now, apply $\gamma$ to the identity
\begin{equation}\label{xy-yx} m_{n,n+1}m_{n+1,n} - m_{n,n-1}m_{n-1,n} = (2z + \m - 1) \cdot 1_n
\end{equation}
to obtain
\begin{equation}\label{gammaeq} m_{n,n+1}m_{n+1,n} - \zeta_{n,n-1}\zeta_{n-1,n}m_{n,n-1}m_{n-1,n} = (2a_nz + 2 b_n+ \m - 1) \cdot 1_n.
\end{equation}
Subtracting these equations yields
\[
(\zeta_{n,n-1}\zeta_{n-1,n} - 1) m_{n,n-1}m_{n-1,n} = (2z - 2 a_n z - 2b_n) \cdot 1_n
\]
and so
\[
(\zeta_{n,n-1}\zeta_{n-1,n} - 1)(z-1)(z+\m - 1)\cdot 1_n= (2z - 2 a_n z - 2b_n) \cdot 1_n.
\]
If $\zeta_{n,n-1}\zeta_{n-1,n} - 1 \neq 0$, then the left-hand side is quadratic in $z$, but the right-hand side is linear $z$. Hence, $\zeta_{n,n-1}\zeta_{n-1,n} = 1$. But now, comparing equations \eqref{xy-yx} and \eqref{gammaeq} means
\[ \gamma_n(z\cdot 1_n) = z \cdot 1_n.
\]
So for all $n$, $\gamma_n$ is the identity on $\bar{A}_{nn}$. Hence, for any $g \in \kk[z] = \bar{A}_{ii}$ and $i,j \in \ZZ$, we have
\[ \gamma( g \cdot m_{ij}) = \zeta_{ij} g \cdot m_{ij}.
\] 
Now, for all $i,j,l \in \ZZ$, $m_{ij}m_{jl} = h m_{il}$ for some $h \in \kk[z] = \bar{A}_{ii}$. By applying $\gamma$, we conclude $\zeta_{ij}\zeta_{jl} = \zeta_{il}$. So if $v \in \ol{A}_{ij}$, we have $\gamma(v) = \zeta_{ij} v = \zeta_{i0}v\zeta_{0j}$, so $\gamma$ is inner by \cite[Theorem 3.10]{sierra}.
\end{proof}

This technical result allows us to prove an analogue of \cite[Corollary 5.6]{sierra}. As in the case of $\gr A_1$, we can check if two autoequivalences of $\gr A$ are naturally isomorphic by checking on a relatively small set of simple modules. 

\begin{lemma} \label{autosimple} Let $\FFF$ and $\FFF'$ be autoequivalences of $\gr A$. Then $\FFF \cong \FFF'$ if and only if $\FFF(S) \cong \FFF'(S)$ for all simple modules $S$ supported at $\ZZ \cup \ZZ - \m$.
\end{lemma}
\begin{proof}Suppose $\FFF(S) \cong \FFF'(S)$ for all simple modules $S$ supported at $\ZZ \cup \ZZ - \m $. By Corollary~\ref{projfactors}, since any rank one projective is determined by its simple factors supported at $\ZZ \cup \ZZ - \m $, we have that $\FFF(P) \cong \FFF'(P)$ for all rank one projectives $P$. In particular, for all integers $n$, $(\FFF')^{-1}\FFF(A \s{n}) \cong A\s{n}$. Hence, $(\FFF')^{-1}\FFF$ is a twist functor. Finally, by Lemma~\ref{inneraut} and Theorem~\ref{twistfunctors}, $(\FFF')^{-1}\FFF \cong \Id_{\gr A}$ and so $\FFF \cong \FFF'$.
\end{proof}

The proof of the preceding lemma also demonstrates what we noticed in section~\ref{sec:proj}---that the structure of $\gr A$ is largely determined by its subcategory of rank one projective modules. We now prove analogues of \cite[Theorem 5.3]{sierra}.

\begin{theorem} \label{cautoequiv} Let $\m \in \NN$ and let $\FFF$ be an autoequivalence of $\gr A$. Then there exist unique integers $a = \pm 1$ and $b$ such that for all $n \in \ZZ$
\[ \{ \FFF(X\s{n}) , \FFF(Y\s{n}) \} \cong \{ X\s{an+b}, Y \s{an+b}\},
\]
and for all $\lambda \in \kk \setminus \ZZ$,
\[ \FFF(M_\lambda) \cong M_{a\lambda + b}.
\]
If $\m > 0$, then additionally
\[ \FFF(Z \s{n}) \cong Z \s{an+b}.
\]
\end{theorem}
\begin{proof} The proof is essentially the same as Sierra's. First, let $\m = 0$. Since $\FFF$ is an autoequivalence, for any integer $n$ there exists an integer $n'$ such that $\FFF$ maps the pair $\{X\s{n}, Y\s{n}\}$ to the pair $\{X\s{n'}, Y\s{n'}\}$ since by Lemmas~\ref{mexts}, these are the only pairs of simple modules with both nonsplit self-extensions as well as nonsplit extensions by each other. Now let $\m > 0$. In this case, for any integer $n$ there exists an integer $n'$ such $\FFF(Z\s{n}) \cong Z\s{n'}$ since the shifts of $Z$ are the only simple modules that have projective dimension $2$. Then we also see that $\FFF$ maps the pair $\{X\s{n}, Y\s{n}\}$ to the pair $\{X\s{n'}, Y\s{n'}\}$ since these are the unique simple modules which have nonsplit extensions with $Z\s{n'}$.

For any $\m \in \ZZ$, for any $\lambda \in \kk \setminus \ZZ$, there exists a $\mu \in \kk \setminus \ZZ$ such that $\FFF(M_\lambda) \cong M_\mu$ since these are the only simple modules whose only nonsplit extensions are self-extensions. Altogether then, there exists a bijective function $g: \kk \sra \kk$ such that
\begin{enumerate}[(i)]
\item If $\lambda \in \ZZ$, then $g(\lambda) \in \ZZ$ and $\FFF(\{X \s{\lambda}, Y \s{\lambda}\}) \cong \{X \s{g(\lambda)}, Y \s{g(\lambda)}\}$ (and if $\m > 0$, $\FFF(Z \s{\lambda}) \cong Z \s{g(\lambda)}$).

\item If $\lambda \notin \ZZ$, then $g(\lambda) \notin \ZZ$ and $\FFF(M_\lambda) \cong M_{g(\lambda)}$.
\end{enumerate}

Now consider the functor $\FFF_0 = \FFF( - \otimes_{\kk[z]} A)_0: \mod \kk[z] \rightarrow \mod \kk[z]$. Notice that $\FFF_0(\kk[z]) \cong \kk[z]$. Further, for all $\lambda \in \kk$, we have that $\FFF_0(\kk[z]/(z+\lambda)) \cong \kk[z]/(z+ g(\lambda))$. If $\lambda \in \ZZ$, this follows from Lemma~\ref{csupport}, otherwise it follows from the definition of $M_\lambda$ and $M_{g(\lambda)}$. The functor $\FFF_0$ gives a $\kk$-algebra homomorphism $\varphi: \Hom_{\kk[z]}(\kk[z], \kk[z]) \sra \Hom_{\kk[z]}(\FFF_0 \kk[z], \FFF_0\kk[z])$. Identify $\kk[z]$ with $\Hom_{\kk[z]}(\kk[z], \kk[z])$, where $h \in \kk[z]$ corresponds to left multiplication by $h$. 

The functor $\FFF_0$ takes the short exact sequence

\[ 0 \sra \kk[z] \overset{(z+\lambda) \cdot}{\ra} \kk[z] \ra \kk[z]/(z+\lambda) \ra 0
\]
to
\[ 0 \sra \FFF_0\kk[z] \ra \FFF_0\kk[z] \ra \kk[z]/(z+g(\lambda)) \ra 0,
\]
so $\varphi$ maps multiplication by $z+\lambda$ to multiplication by $c(z+g(\lambda))$ for some $c \in \kk^*$. Therefore, $\varphi(z)$ must be linear in $z$, i.e. $\varphi(z) = \gamma z + \delta$ for some $\gamma, \delta \in \kk$. Then 
\[ \gamma z + \delta + \lambda = \varphi(z + \lambda) = c(z+g(\lambda))
\]
and so $g(\lambda) = (\lambda + \delta)/\gamma$. Since $g$ maps $\ZZ$ bijectively to $\ZZ$, we conclude that $\gamma = \pm 1$ and $\delta \in \ZZ$. Take $a = \gamma$ and $b = a \delta$.
\end{proof}

\begin{theorem} \label{dautoequiv} Let $\m \in \kk \setminus \ZZ$, and $\m \notin \ZZ + 1/2$. Let $\FFF$ be an autoequivalence of $\gr A$. Then exactly one of the following is true:
\begin{enumerate}
\item There exists a unique integer $b$ such that for all $n \in \ZZ$
\[ \{\FFF(X_0\s{n} ), \FFF(Y_0 \s{n}) \} \cong \{X_0\s{ n + b}, Y_0 \s{n+b} \},
\]
\[ \{\FFF(X_\m\s{n} ), \FFF(Y_\m \s{n}) \} \cong \{X_\m\s{ n + b}, Y_\m \s{n+b} \},
\]
and for all $\lambda \in \kk \setminus ( \ZZ \cup \ZZ + \m)$
\[ \FFF(M_\lambda) \cong M_{\lambda + b}.
\]

\item There exists a unique integer $b$ such that for all $n \in \ZZ$
\[ \{\FFF(X_0\s{n} ), \FFF(Y_0 \s{n}) \} \cong \{X_\m\s{-n + b}, Y_\m \s{-n+b} \} ,
\]
\[\{\FFF(X_\m\s{n} ), \FFF(Y_\m \s{n}) \} \cong \{X_0\s{-n + b}, Y_0 \s{-n+b} \},
\]
and for all $\lambda \in \kk \setminus ( \ZZ \cup \ZZ + \m)$
\[ \FFF(M_\lambda) \cong M_{-\lambda + \m + b}.
\]
\end{enumerate}
\end{theorem}
\begin{proof}This proof is quite similar to the previous one, although the details are slightly messier.

From Lemma~\ref{simples} and Lemma~\ref{dexts}, for any $\lambda \in \kk \setminus (\ZZ \cup \ZZ + \m )$, there exists a $\mu \in \kk \setminus (\ZZ \cup \ZZ + \m )$ such that $\FFF(M_\lambda) \cong M_\mu$, since these are the only simple modules with nonsplit self extensions. Further, by Lemma~\ref{dexts}, for all $n \in \ZZ$ $\FFF$ must map the pair $\{X_0\s{n}, Y_0 \s{n}\}$ to either a pair $\{X_0\s{n'}, Y_0 \s{n'}\}$ or a pair $\{X_\m\s{n'}, Y_\m \s{n'}\}$ for some integer $n'$, since these pairs form the only nonsplit extensions of two nonisomorphic simples. Likewise for the pair $\{X_\m\s{n}, Y_\m \s{n}\}$. That is, there is a bijection $g: \kk \rightarrow \kk$ such that
\begin{enumerate}[(1)]
\item If $\lambda \in \ZZ$, then either
\begin{itemize} 
\item $g(\lambda) \in \ZZ$ and $\FFF(\{X_0 \s{\lambda}, Y_0 \s{\lambda}\}) \cong \{X_0 \s{g(\lambda)}, Y_0 \s{g(\lambda)}\}$ or else
\item $g(\lambda) \in \ZZ +\m$ and $\FFF(\{X_0 \s{\lambda}, Y_0 \s{\lambda}\}) \cong \{X_\m \s{g(\lambda) - \m}, Y_\m \s{g(\lambda) - \m}\}$.
\end{itemize}

\item If $\lambda \in \ZZ + \m$, then either
\begin{itemize}
\item $g(\lambda) \in \ZZ +\m$ and $\FFF(\{X_\m \s{\lambda - \m}, Y_\m \s{\lambda - \m}\}) \cong \{X_\m \s{g(\lambda) - \m}, Y_\m \s{g(\lambda) - \m}\}$ or else 
\item$g(\lambda) \in \ZZ$ and $\FFF(\{X_\m \s{\lambda - \m}, Y_\m \s{\lambda - \m}\}) \cong \{X_0 \s{g(\lambda)}, Y_0 \s{g(\lambda)}\}$.
\end{itemize}

\item If $\lambda \notin (\ZZ \cup \ZZ + \m)$, then $g(\lambda) \notin (\ZZ \cup \ZZ + \m)$ and $\FFF(M_\lambda) \cong M_{g(\lambda)}$.
\end{enumerate}

As in the proof of Lemma~\ref{cautoequiv}, consider the functor $\FFF_0 = \FFF( - \otimes_{\kk[z]} A)_0: \mod \kk[z] \rightarrow \mod \kk[z]$. Just as in the previous proof, there exist $\delta, \gamma \in \kk$ such that $g(\lambda) = (\lambda + \delta)/\gamma$. Since $g$ maps $\ZZ \cup \ZZ + \m$ bijectively to itself, we conclude that $\gamma = \pm 1$. Since we assumed $\m \notin \ZZ +1/2$, if $\gamma = 1$ then $\delta \in \ZZ$, and if $\gamma = -1$ then $\delta \in \ZZ - \m$. 

If $\gamma = 1$, let $b = \delta \in \ZZ$. In this case,
\begin{align*}\{\FFF(X_0\s{n} ), \FFF(Y_0 \s{n}) \} &\cong \{X_0\s{ n + b}, Y_0 \s{n+b} \} \mbox{ and } \\ \{\FFF(X_\m\s{n} ), \FFF(Y_\m \s{n}) \} &\cong \{X_\m\s{ n + b}, Y_\m \s{n+b} \}.
\end{align*}

If $\gamma = -1$, let $b = - \m - \delta$. In this case,
\begin{align*}\{\FFF(X_0\s{n} ), \FFF(Y_0 \s{n}) \} &\cong \{X_\m\s{-n + b}, Y_\m \s{-n+b} \} \mbox{ and }\\
 \{\FFF(X_\m\s{n} ), \FFF(Y_\m \s{n}) \} &\cong \{X_0\s{-n + b}, Y_0 \s{-n+b} \}. \qedhere
\end{align*}
\end{proof}

\begin{definition}[{\cite[Definition 5.4]{sierra}}] \label{ranksign} If $\FFF$ is an autoequivalence of $\gr A$, we call the integer $b$ above the \emph{rank} of $\FFF$. The integer $a$ above is called the \emph{sign} of $\FFF$. If $a = 1$, then we say $\FFF$ is \emph{even} and if $a = -1$, we say $\FFF$ is \emph{odd}. If $\FFF$ is even and has rank $0$, we say that $\FFF$ is \emph{numerically trivial}.
\end{definition}

\begin{example} As was the case for autoequivalences of $\gr A_1$, $\SSS_A^n$ is an even autoequivalence of rank $n$ and $\omega$ is an odd autoequivalence of rank $-1$. 
\end{example}

Notice that if $\m \in \ZZ + 1/2$, then there are potentially extra symmetries of $\gr A$. In the proof of Theorem~\ref{dautoequiv}, the function $g: \kk \sra \kk$ could be of the form $g(\lambda) = \lambda + 1/2$, as this maps $\ZZ + \m $ bijectively to itself if $\m \in \ZZ +1/2$. We show that in fact, in this case, there is an autoequivalence which acts in this way.

\begin{figure}[h!]
\centerline{
	\begin{tikzpicture}[ scale=1.6]	
		\draw[thick,<-](-3.9,0)--(-3.1,0);
		\fill[black](-3,-.1) circle (1pt);
		\fill[black](-3,.1) circle (1pt);
		\node at (-3,-.1)[label = below: $-3$]{};
		\draw[thick](-2.9,0)--(-2.6,0);
				\fill[black](-2.5,-.1) circle (1pt);
		\fill[black](-2.5,.1) circle (1pt);
		\node at (-2.5,.1)[label = above: $-5/2$]{};
		\draw[thick](-2.4,0)--(-2.1,0);
		\fill[black](-2,-.1) circle (1pt);
		\fill[black](-2,.1) circle (1pt);
		\node at (-2,-.1)[label = below: $-2$]{};
		\draw[thick](-1.9,0)--(-1.6,0);
				\fill[black](-1.5,-.1) circle (1pt);
		\fill[black](-1.5,.1) circle (1pt);
		\node at (-1.5,.1)[label = above: $-3/2$]{};
		\draw[thick](-1.4,0)--(-1.1,0);
		\fill[black](-1,-.1) circle (1pt);
		\fill[black](-1,.1) circle (1pt);
		\node at (-1,-.1)[label = below:$-1$]{};
		\draw[thick](-.9,0)--(-.6,0);
			\fill[black](-.5,-.1) circle (1pt);
		\fill[black](-.5,.1) circle (1pt);
		\node at (-.5,.1)[label = above:$-1/2$]{};
		\draw[thick](-.4,0)--(-.1,0);
		\fill[black](0,-.1) circle (1pt);
		\fill[black](0,.1) circle (1pt);
		\node at (0,-.1)[label = below: $0$]{};
		\draw[thick](.1,0)--(.4,0);
			\fill[black](.5,-.1) circle (1pt);
		\fill[black](0.5,.1) circle (1pt);
		\node at (0.5,.1)[label = above: $1/2 $]{};
		\draw[thick](.6,0)--(.9,0);
		\fill[black](1,-.1) circle (1pt);
		\fill[black](1,.1) circle (1pt);
		\node at (1,-.1)[label = below: $1$]{};
			\draw[thick](1.1,0)--(1.4,0);		
		\fill[black](1.5,-.1) circle (1pt);
		\fill[black](1.5,.1) circle (1pt);
		\node at (1.5,.1)[label = above: $3/2$]{};
		\draw[thick](1.6,0)--(1.9,0);		
		\fill[black](2,-.1) circle (1pt);
		\fill[black](2,.1) circle (1pt);
		\node at (2,-.1)[label = below: $2$]{};
		\draw[thick,](2.1,0)--(2.4,0);	
			\fill[black](2.5,-.1) circle (1pt);
		\fill[black](2.5,.1) circle (1pt);
		\node at (2.5,.1)[label = above: $5/2$]{};
		\draw[thick,](2.6,0)--(2.9,0);		
		\fill[black](3,-.1) circle (1pt);
		\fill[black](3,.1) circle (1pt);
		\node at (3,-.1)[label = below: $3$]{};
			\draw[thick,](3.1,0)--(3.4,0);	
		\fill[black](3.5,-.1) circle (1pt);
		\fill[black](3.5,.1) circle (1pt);
		\node at (3.5,.1)[label = above: $7/2$]{};
		\draw[thick,->](3.6,0)--(3.9,0);		
	\end{tikzpicture}
	}
\caption{The simple modules of $\gr A(f)$ when $\m = 1/2$.}
\label{alphahalf}
\end{figure}

\begin{proposition} \label{halfshift} Let $\m \in \ZZ + 1/2$. There exists an autoequivalence $\FFF$ of $\gr A$ such that for all $n \in \ZZ$,
\begin{align*} &\FFF(X_0\s{n}) \cong X_{\m } \s{ n +1/2 - \m }, & &\FFF(Y_0 \s{n}) \cong Y_{\m } \s{n + 1/2 - \m },\\
&\FFF(X_{\m }\s{n} ) \cong X_0 \s{ n + 1/2 + \m }, & &\FFF(Y_{\m } \s{n}) \cong Y_0 \s{n+ 1/2 + \m},
\end{align*}
and for all $\lambda \in \kk \setminus ( \ZZ \cup \ZZ + 1/2)$
\[ \FFF(M_\lambda) \cong M_{\lambda + 1/2}.
\]
\end{proposition}

\begin{proof} We will construct an autoequivalence $\FFF$ that translates Figure~\ref{alphahalf} by 1/2. By Lemmas~\ref{projfunc2} and \ref{projfunc}, it suffices to define $\FFF$ on the full subcategory $\RRR$ of $\gr A$ consisting of the canonical rank one projective right $A$-modules. Let $P$ be a rank one canonical projective module with structure constants $\{c_i\}$. By the work in section~\ref{sec:ideals}, we know we can write each structure constant as $c_i = a_i b_i$ where $a_i \in \{1, \sigma^{i}(z)\}$ and $b_i \in \{1, \sigma^{i}(z+\m ) \}$. Let $\sigma^{1/2}$ be the automorphism of $\kk[z]$ with $\sigma^{1/2}(z) = z + 1/2$. 

Define $\FFF(P)$ to be the canonical rank one projective module whose structure constants $\{c'_i\}$ are defined as follows. If $a_{n + \m - 1/2} = \sigma^{n + \m - 1/2}(z)$, then $c'_n$ has a factor of $\sigma^{n}(z + \m )$. If $b_{n-\m - 1/2} = \sigma^{n - \m - 1/2}(z + \m )$, then $c'_n$ has a factor of $\sigma^{n}(z)$. Overall, the irreducible factors of the $\{c_i'\}$ are given by $\{\sigma^{1/2}(a_i)\}$ and $\{\sigma^{1/2}(b_i)\}$, where the factors appear in the structure constant of the appropriate degree. Since, for all $n$, $c'_n \in \{1, \sigma^{n}(z), \sigma^{n}(z+1/2), \sigma^{n}(f)\}$, and for $n \gg 0$, $c_n' = 1$ and $c_{-n}' = \sigma^{-n}(f)$, these structure constants define a canonical rank one projective module. 

Let $P$ and $Q$ be canonical rank one projective modules with structure constants $\{c_i\}$ and $\{d_i\}$, respectively. Let the structure constants of $\FFF(P)$ and $\FFF(Q)$ be $\{c_i'\}$ and $\{d_i'\}$ respectively. By Lemma~\ref{projmaxemb} and Corollary~\ref{dmaxemb}, there is an $N \in \ZZ$ such that $\uHom_A(P,Q)$ and $\uHom_A(\FFF(P), \FFF(Q))$ are generated as a $\kk[z]$-module by multiplication by 
\begin{align*}&\theta_{P,Q} = \frac{\prod_{j \geq N}d_j}{\prod_{j \geq N}\gcd\left( c_j, d_j\right)} \quad \mbox{and}\\
& \theta_{\FFF(P), \FFF(Q)} = \frac{\prod_{j \geq N}d'_j}{\prod_{j \geq N}\gcd\left( c'_j, d'_j\right)} ,
\end{align*}
respectively. Notice that by the way we defined the structure constants of $\FFF(P)$ and $\FFF(Q)$, 
\[\theta_{\FFF(P), \FFF(Q)} = \sigma^{1/2}\left(\theta_{P,Q}\right).\]
Let $g \in \uHom_A(P,Q)$, so $g = \varphi \theta_{P,Q}$ for some $\varphi \in \kk[z]$. Define $\FFF(g)$ to be left multiplication by $\sigma^{1/2}(\varphi \theta_{P,Q}) = \sigma^{1/2}(\varphi)\theta_{\FFF(P),\FFF(Q)} $.

Since every morphism in $\RRR$ is given by left multiplication by an element of $\kk[z]$ and $\FFF$ acts on morphisms by applying $\sigma^{1/2}$ to this element, clearly $\FFF$ is functorial. Since the identity morphism is just multiplication by 1, we have $\FFF(\Id_P) = \Id_{\FFF(P)}$. Hence, $\FFF$ is a functor. It is also easy to see that $\FFF$ is essentially surjective on canonical rank one projective modules. Given a rank one projective module $P'$, we reverse the structure constant construction to construct a canonical rank one projective $P$ such that $\FFF(P) \cong P'$. Since $\sigma^{1/2}$ is an automorphism of $\kk[z]$, it gives an isomorphism
\[ \uHom_A(P,Q) = \theta_{P,Q} \kk[z] \cong \sigma^{1/2}(\theta_{P,Q} \kk[z]) = \uHom_A\left(\FFF(P), \FFF(Q)\right).
\]
Hence, $\FFF$ is full and faithful and so is an autoequivalence of $\RRR$, which extends uniquely to an autoequivalence of $\gr A$ by Corollary~\ref{subequiv} and Lemma~\ref{projfunc2}.

We need only show that $\FFF$ has the claimed action on simple modules. For each $\lambda \in \kk$, $\FFF$ maps the exact sequence
\[ 0 \ra A \overset{(z+\lambda) \cdot}{\ra} A \ra A/(z+\lambda)A \ra 0
\]
to the exact sequence
\[ 0 \ra \FFF(A) \overset{(z+\lambda + 1/2) \cdot}{\ra} \FFF(A) \ra \FFF(A/(z+\lambda)A) \ra 0.
\]
Hence, $\FFF(A/(z+\lambda)A)$ is supported at $- (\lambda+1/2)$ so if $\lambda \notin \ZZ +1/2$, $\FFF(M_\lambda) \cong M_{\lambda + 1/2}$. The pair of simple modules $\{X_0\s{n}, Y_0\s{n}\}$ which are supported at $-n$ must be mapped to the pair $\{X_{\m }\s{n +1/2 - \m }, Y_{\m }\s{n + 1/2 - \m }\}$, as these are the simples supported at $-(n+1/2)$. Similarly, the pair $\{X_{\m }\s{n}, Y_{\m } \s{n}\}$ must map to the pair $\{X_{0} \s{n - 1/2 + \m }, Y_{0} \s{n - 1/2 + \m } \}$.

Now consider the short exact sequence
\[ 0 \ra (xA + zA)\s{n} \overset{1 \cdot}{\ra} A\s{n} \sra X_0\s{n} \ra 0. \label{Xsesii}
\]
Using the construction in Lemmas~\ref{cssc} and \ref{dssc}, we can explicitly compute the structure constants of $xA + zA$. Since in each graded component of degree $i \leq 0$, $(xA+zA)_i = z(A)_i$ and for $i> 0$, $(xA+zA)_i = A_i$, therefore $xA+zA$ has structure constants which are the same as that of $A$, except in degree $0$ where $xA+zA$ has a structure constant of $z$. Multiplying structure constants to compute the canonical representation of $xA + zA$, we observe that $xA + zA$ is itself a canonical rank one projective module. Since maximal embeddings of canonical rank one projectives are given by multiplication by elements in $\kk[z]$, the inclusion $xA + zA \sra A$ in \eqref{Xsesii} is a maximal embedding. 

Since $\FFF$ of multiplication by 1 is again given by multiplication by 1, $\FFF$ maps \eqref{Xsesii} to the exact sequence
\[0 \ra \FFF((xA + zA) \s{n}) \overset{1 \cdot}{\ra} \FFF(A \s{n}) \ra \FFF(X_0 \s{n}) \ra 0.
\] 
Hence, $\FFF(X_0 \s{n})$ is zero in sufficiently large degree, so $\FFF(X_0\s{n}) \cong X_{\m }\s{n + 1/2 - \m }$. This then implies $\FFF(Y_0 \s{n}) \cong Y_{\m }\s{n+1/2 - \m }$. A similar computation for $X_{\m }\s{n}$ completes the proof.
\end{proof}

\begin{remark}The autoequivalence constructed in Proposition~\ref{halfshift} translates the picture of the simple modules by $1/2$. Since the square of this autoequivalence is isomorphic to $\SSS_A$, we call it $\SSS^{1/2}$.
\end{remark}

For completeness, we can extend Definition~\ref{ranksign} to the case $\m \in \ZZ + 1/2$ in a natural way. If $\FFF$ is an autoequivalence of $\gr A(f)$, and for all $\beta \in \kk$, $\FFF$ maps the simples supported at $\beta$ to simples supported at $\beta + n/2$ for some $n \in \ZZ$, we say that $\FFF$ is \emph{even} and has \emph{rank} $n/2$. If $\FFF$ maps the simples supported at $\beta$ to simples supported at $-\beta + n/2$ for some $n \in \ZZ$, we say that $\FFF$ is \emph{odd} and has \emph{rank} $n/2$. If $\FFF$ is even and has rank 0 then we say $\FFF$ is \emph{numerically trivial}. The autoequivalence $\SSS^{1/2}$ is even and has rank $1/2$. 

\begin{corollary} \label{cpiccor} Let $\Pic_0(\gr A)$ be the subgroup of $\Pic(\gr A)$ of numerically trivial autoequivalences. Then $\Pic(\gr A) \cong \Pic_0( \gr A) \rtimes D_\infty$.
\end{corollary}
\begin{proof} By Lemma~\ref{autosimple}, each autoequivalence in $\Pic(\gr A)$ is determined by its action on the simple modules supported at $\ZZ \cup \ZZ - \m$. Let $\m \not \in \ZZ + 1/2$. By checking their action on the simple modules supported at $\ZZ \cup \ZZ - \m$, we observe that $\omega \SSS_A \cong \SSS_A^{-1} \omega$ and so the subgroup
\[ \langle \omega, \SSS_A \rangle \subseteq \Pic(\gr A)
\]
is isomorphic to $D_\infty$. 

Again, by considering the action of numerically trivial autoequivalences on the simple modules, we observe that $\Pic_0(\gr A)$ is a normal subgroup of $\Pic(\gr A)$. If $\FFF$ is a numerically trivial autoequivalence in $\langle \omega, \SSS_A \rangle $, then we can write $\FFF = \SSS_A^{i} \omega^j$ for some $i, j \in \ZZ$. Since $\FFF$ is numerically trivial and $\omega^2 = \Id_{\gr A}$, in fact $\FFF = \SSS_A^{i}$ which then implies $i = 0$ so $\FFF = \Id_{\gr A}$. Therefore, 
\[ \Pic_0(\gr A) \cap \langle \omega, \SSS_A \rangle = \{ \Id_{\gr A}\}.
\]

By Theorems~\ref{cautoequiv} and \ref{dautoequiv}, any autoequivalence can be written as the product of an autoequivalence in $\langle \omega, \SSS_A \rangle$ and a numerically trivial autoequivalence. Therefore, 
\[\Pic(\gr A) \cong \Pic_0(\gr A) \rtimes D_\infty,\]
as desired.

In the case $\m \in \ZZ + 1/2$, $\Pic(\gr A)$ contains a subgroup isomorphic to $D_\infty$ generated by the autoequivalences $\SSS^{1/2}$ and $\omega$. This is a finer copy of $D_\infty$, containing $\langle \omega, \SSS_A\rangle$ as a subgroup. The remainder of the proof is identical to the previous case.
\end{proof}

\subsection{Involutions of $\gr A$} \label{sec:iotas}

Having shown that $\gr A(f)$ has the same rigidity as $\gr A_1$, we now show that $\Pic_0(\gr A(f))$ is also isomorphic to $\Pic_0(\gr A_1)$. We construct autoequivalences which are analogous to the involutions $\iota_j$ of Sierra.

\begin{proposition} \label{ciota}Let $\m \in \NN$. Then for any $j \in \ZZ$, there is a numerically trivial autoequivalence $\iota_j$ of $\gr A$ such that $\iota_j(X \s{j}) \cong Y \s{j}$, $\iota_j(Y\s{j}) \cong X \s{j}$, and $\iota_j(S) \cong S$ for all other simple modules $S$. For any $i, j \in \ZZ$, $\SSS_A^i \iota_j \cong \iota_{i+j}\SSS_A^i$, and $\iota_j^2 \cong \Id_{\gr A}$.
\end{proposition}
\begin{proof}First, suppose $\m \in \NNp$. It will suffice to construct $\iota_0$, for if $\iota_0$ exists then we may define $\iota_j = \SSS_A^j \iota_0 \SSS_A^{-j}$. Let $\PPP$ be the full subcategory of $\gr A$ whose objects are direct sums of canonical rank one projective modules and let $\RRR$ be the full subcategory of $\gr A$ whose objects are the canonical rank one projective modules. Note that since every finitely generated graded right $A$-module has a projective resolution by objects in $\PPP$, by Lemma~\ref{projfunc}, we can construct $\iota_0$ by defining it on $\PPP$, then extend to $\gr A$. By Lemma~\ref{projfunc2}, it suffices to define $\iota_0$ only on $\RRR$.

Let 
\[S = \{0, X, Y, Z, E_{Z,X}, E_{Z,Y}, E_{X,Z}, E_{Y,Z}, E_{Z,Y,X}, E_{X,Z,X}, E_{Y,Z,Y}\}\]
 and let $\DDD$ be the full subcategory of $\gr A$ whose objects are exactly the elements of $S$. It is clear that $\DDD$ is closed under subobjects. Let $P$ be a graded rank one projective module. By Corollary~\ref{projXYZ}, $P$ surjects onto exactly one of $X$, $Y$, and $Z$, a module that we call $F_0(P)$. Suppose $F_0(P) = X$. Since $P$ is projective, the surjection $f_0: P \sra X$ lifts to a morphism $f_1: P \sra E_{X,Z}$. Since $f_1$ is a lift of a surjection to $X$, $f_1$ is surjective, as $E_{X,Z}$ has no subobject isomorphic to $X$. Again, since $P$ is projective, $f_1$ lifts to a surjection $f_2: P \sra E_{X,Z,X}$, as $E_{X,Z,X}$ has no subobject isomorphic to $E_{X,Z}$. 

Given the structure constants $\{c_i\}$ of $P$, we can in fact construct the submodule of $P$ that is the kernel of $f_2$. We construct this submodule in three steps. First, we construct $K_0$, the kernel of $f_0$, which is unique by Corollary~\ref{homIS}. Following the construction in Lemma~\ref{cssc}, $K_0$ is the submodule of $P$ that has structure constants equal to $c_i$ for all $i \in \ZZ$ except when $ i= - \m $, where $K_0$ has structure constant $z c_{- \m }$. We then construct $\ker f_1 = K_1$ as a submodule of $K_0$. Note that since $(P/K_1) / (K_0/K_1) \cong P/K_0$, we must have that $K_1$ is the unique submodule of $K_0$ which is the kernel of the surjection $K_0 \sra Z$. Again, by the construction in Lemma~\ref{cssc}, $K_1$ has structure constants $c_i$ for all $i \in \ZZ$ except when or $i=0$ where $K_1$ has structure constant $zc_0$. Finally, we can construct $\ker f_2 = K_2$ as a submodule of $K_1$, by constructing the unique submodule such that $K_1/K_2 \cong X$. Observe that $K_2$ has structure constants $c_i$ for all $i \in \ZZ$, except when $i = - \m$, where $K_2$ has structure constant $z c_{- \m }$ and when $i= 0$, where $K_2$ has structure constant $z c_0$.

Similarly, if $F_0(P) = Y$, there is a unique submodule of $P$ which is the kernel of a surjection $P \sra E_{Y,Z,Y}$. If $F_0(P) = Z$, then $\Hom_{\gr A}(P,E_{Z,Y,X}) = \kk$ so there is a unique submodule which is the kernel of a surjection $P \sra E_{Z,Y,X}$. In any case, there exists a unique smallest submodule $N$ of $P$ such that $P/N \in \DDD$. Define $\iota_0 P = N$. By Lemma~\ref{projfunc3}, $\iota_0$ gives an additive functor $\RRR \sra \gr A$ such that $\iota_0$ acts on morphisms by restriction. By using Lemma~\ref{projfunc2} and Proposition~\ref{projfunc}, we extend $\iota_0$ to a functor $\iota_0: \gr A \sra \gr A$.

We now show that $\iota_0$ has the claimed properties. Suppose $P$ has structure constants $\{c_i\}$. Above, we computed the structure constants, $\{d_i\}$ of $\iota_0 P$. By Lemma~\ref{cssc} and Corollary~\ref{cprojsc} if $F_0(P) = X$, then $c_{-\m } \in \{1, \sigma^{-\m }(z)\}$ and $c_0 \in \{1, z+\m \}$. We showed that $d_{-\m } = z c_{-\m }$ and $d_0 = z c_0 \in \{z, f\}$. For all $i \neq 0, -\m $, we saw that $c_i = d_i$. Hence, we can find $\iota_0(P)$ explicitly as a submodule of $P$ as follows: 
\[ (\iota_0 P)_i = \begin{cases} z^2 P_i & \mbox{ if } i \leq -\m \\
 zP_i & \mbox{ if } -\m < i \leq 0 \\
P_i & \mbox{ if } i >0.
\end{cases}
\]
Similarly, if $F_0(P) = Y$, then $c_0 \in \{z,f\}$ and $c_{-\m } \in \{ \sigma^{-\m }(z+\m ), \sigma^{-\m }(f)\}$ and $d_0 = z^{-1}c_0$ and $d_{- \m } = z^{-1} c_{-\m }$. We can explicitly construct $\iota_0P$ as follows: 
\[(\iota_0 P)_i = \begin{cases} P_i & \mbox{ if } i \leq -\m \\
 zP_i & \mbox{ if } -\m < i \leq 0 \\
z^2P_i & \mbox{ if } i >0.
\end{cases}
\]
Finally, if $F_0(P) = Z$, then $P$ surjects onto $E_{Z,Y,X} = A/zA$, and $(\iota_0 P)_i = z P_i$ for all $i \in \ZZ$, so $c_i = d_i$ for all $i \in \ZZ$.

We can describe the action of $\iota_0$ on $P$ purely in terms of its structure constants, as follows. If both $c_0$ and $c_{-\m }$ can be multiplied by $z$ (i.e. for $i = 0, -\m $, $zc_i \in \{ 1, \sigma^{i}(z), \sigma^{i}(z+\m ), \sigma^{i}(f)\}$) then $\iota_0 P$ has $d_0 = z c_0$ and $d_{-\m } = zc_{-\m }$. Likewise, if both $c_0$ and $c_{-\m }$ can be divided by $z$, then $\iota_0 P$ has $d_0 = z^{-1}c_0$ and $d_{-\m } = z^{-1}c_{-\m }$. Otherwise, $d_0 = c_0$ and $d_{-\m } = c_{-\m }$. Observe that by Lemma~\ref{cssc}, if $F_0(P) = X$ then $F_0(\iota_0 P) = Y$ and vice versa. Hence, by repeating the above process once by taking the kernel to $E_{X,Z,X}$ and next the kernel to $E_{Y,Z,Y}$ we compute that $\iota_0^2 P = z^2P$. 

So for any rank one projective, $P$, $\iota_0^2P = z^2P \cong P$. Additionally, if $P'$ is another rank one projective, Proposition~\ref{projhom} tells us that $\Hom_{\gr A}(z^2P, z^2P') = \Hom_{\gr A}(P, P')$ is given by left multiplication by a $\kk[z]$-multiple of some $\theta \in \kk(z)$. Since $\iota_0$ is defined on morphisms to be restriction, this shows that $\iota_0^2$ also gives an isomorphism $\Hom_{\gr A}(P, P') \cong \Hom_{\gr A}(\iota_0^2 P , \iota_0^2 P')$. Since $\iota_0$ is additive, it preserves finite direct sums, so for a direct sum of rank one projectives, $\iota_0^2$ is given by multiplication by $z^2$ in each component. Hence, $\iota_0^2$ is naturally isomorphic to the identity functor on the full subcategory of finite direct sums of rank one projectives. Extending to all of $\gr A$, this shows that $\iota_0$ is an autoequivalence of $\gr A$ (with quasi-inverse $\iota_0$).

Now, because for any rank one graded projective module $P$, the structure constants of $\iota_0 P$ differ from those of $P$ only in degrees $0$ and $-\m $, $\iota_0 P$ has the same integrally supported simple factors as $P$ except possibly $X$, $X\s{\m }$, $Y$, $Y \s{-\m }$, $Z$, and $Z \s{\pm \m}$. But if $F_{\m }(P) = X \s{\m }$, then by Table~\ref{tablestruc}, $c_0 \in \{1, z\}$. By the construction of $\iota_0 P$ (multiplying $c_0$ by $1$, $z$ or $z^{-1}$ to compute $d_0$), this means $d_0 \in \{1,z\}$ as well, so $F_{\m }(\iota_0 P) = X\s{\m }$. Similarly, if $F_{-\m }(P) = Y\s{-\m}$ then $F_{-\m }(\iota_0 P) = Y\s{-\m }$ as well. If $F_0(P) = Z$, then we observed above that $\iota P \cong P$ so $F_0(\iota_0 P) = Z$. If $F_{\m }(P) = Z\s{\m }$ then $c_0 \in \{z+\m , f\}$. Again, since $\iota_0 P$ acts on structure constants only multiplying by $1$, $z$, or $z^{-1}$, this means $d_0 \in \{z+\m , f\}$. Since $\iota_0 P$ does not affect the structure constant in degree $\m $, by Lemma~\ref{cssc}, $F_{\m }(\iota_0 P) = Z\s{\m }$. Similarly, if $F_{-\m }(P) = Z\s{-\m }$, then $F_{-\m }(\iota_0 P) = Z\s{-\m }$. Altogether, all of the integrally supported simple factors of $\iota_0 P$ are the same as those of $P$, except that if $F_0(P) = X$ then $F_0(\iota_0 P) = Y$ and vice versa.

We will now show that $\iota_0$ fixes all simples modules other than $X$ and $Y$. Let $S \not \in \{X, Y\}$ be an integrally supported simple module. Suppose for contradiction that $\iota_0 S \not \cong S$. By Lemma~\ref{anyproj}, we can construct a canonical rank one projective module $P$ such that $S$ is a factor of $P$ but $\iota_0 S$ is not. Note that $\iota_0 S$ is a factor of $\iota_0 P$. But by the discussion above, $P$ and $\iota_0 P$ have the same integrally supported simple factors except possibly $X$ and $Y$, so $S$ is also a factor of $\iota_0 P$. Applying $\iota_0$ again, we conclude that both $\iota_0 S$ and $\iota_0^2 S$ are factors of $\iota_0^2 P$, but since $\iota_0^2 \cong \Id_{\gr A}$, this is a contradiction, as $\iota_0 S$ is not a factor of $P$. Hence, $\iota_0$ fixes all integrally supported simple modules other than $X$ and $Y$.

Since $\iota_0$ fixes all integrally supported simple modules other than $X$ and $Y$, $\iota_0$ is numerically trivial. By Theorem~\ref{cautoequiv}, $\iota_0 M_\lambda \cong M_\lambda$ for all $\lambda \in \kk \setminus \ZZ$. Further, for a projective module $P$, if $F_0(P) = X$ then $F_0(\iota_0 P) = Y$. Therefore, $\iota_0 X \cong Y$.

For the case that $\m = 0$, we use the same argument, letting $\DDD$ be the full subcategory of $\gr A$ whose objects are in the set $S = \{0, X, Y, E_{X,X}, E_{Y,Y}\}$. In this case, for a rank one projective $P$, we can again explicitly construct $\iota_0 P \subseteq P$. If $F_0(P) = X$, then
\[ (\iota_0 P)_i = \begin{cases} z^2 P_i & \mbox{ if } i \leq 0 \\
P_i & \mbox{ if } i >0
\end{cases}
\]
and if $F_0(P) = Y$, then 
\[(\iota_0 P)_i = \begin{cases} P_i & \mbox{ if } i \leq 0 \\
z^2P_i & \mbox{ if } i >0.
\end{cases}
\]
Let $\{d_i\}$ be the structure constants of $\iota_0 P$. In this case, if $c_0 = 1$ then $d_0 = z^2$. If $c_0 = z^2$ then $d_0 = 1$. The remainder of the proof is analogous to the previous case.
\end{proof}

In the case that $\m \in \NN$, we have thus constructed analogues of Sierra's autoequivalences $\iota_j$ \cite[Proposition 5.7]{sierra}. These autoequivalences, which we also call $\iota_j$ share many of the same properties as Sierra's. First notice that since we defined $\iota_j = \SSS^j_A \iota_0 \SSS^{-j}_A$, we construct $\iota_j$ by shifting all the modules of $S$ by $j$ and repeating the same construction. This also that means that for a rank one graded projective module $P$, $\iota_j^2 P = (z+j)^2 P$, by the same argument as in the previous proof. Also, reviewing the construction above, it is clear that for any integers $i$ and $j$, $\iota_i \iota_j = \iota_j \iota_i$ and so the subgroup of $\Pic(\gr A)$ generated by the $\{\iota_j\}$ is isomorphic to $(\ZZ/2\ZZ)^{(\ZZ)}$. 

We identify $(\ZZ/2\ZZ)^{(\ZZ)}$ with finite subsets of the integers, $\ZZ_\fin$ with operation given by exclusive or. We often denote the singleton set $\{n\} \in \ZZ_\fin$ as simply $n$. For each $J \in \ZZ_\fin$, we define the autoequivalence
\[ \iota_J = \prod_{j \in J} \iota_j.
\]
As we noted in the previous proof, $\iota_j$ has quasi-inverse $\iota_j$ so $\iota_J$ has quasi-inverse $\iota_J$. For completeness, define $\iota_\emptyset = \Id_{\gr A}$.

Having constructed involutions in the case $\m \in \NN$, we now turn our attention to the case $\m \in \kk \setminus \ZZ$. In this case we will also be able to construct generalizations of Sierra's autoequivalences.

\begin{proposition}\label{diota}Let $\m \in \kk \setminus \ZZ$. Then for any $j \in \ZZ$, there is a numerically trivial autoequivalence $\iota_{(j, \emptyset)}$ of $\gr A$ such that $\iota_{(j, \emptyset)}(X_0 \s{j}) \cong Y_0 \s{j}$, $\iota_{(j, \emptyset)}(Y_0 \s{j}) \cong X_0\s{j}$, and $\iota_{(j, \emptyset)}(S) \cong S$ for all other simple modules $S$. For any $i,j \in \ZZ$, $\SSS_A^i \iota_{(j, \emptyset)} \cong \iota_{(i+j, \emptyset)}\SSS_A^i$, and $(\iota_{(j, \emptyset)})^2 \cong \Id_{\gr A}$.

Similarly, for any $j \in \ZZ$, there is a numerically trivial autoequivalence $\iota_{(\emptyset, j)}$ of $\gr A$ such that $\iota_{(\emptyset, j)}(X_\m \s{j}) \cong Y_\m \s{j}$, $\iota_{(\emptyset, j)}(Y_\m \s{j}) \cong X_\m\s{j}$, and $\iota_{(\emptyset, j)}(S) \cong S$ for all other simple modules $S$. For any $i,j \in \ZZ$, $\SSS_A^i \iota_{(\emptyset, j)} \cong \iota_{(\emptyset, i+j)}\SSS_A^i$, and $(\iota_{(\emptyset, j)})^2 \cong \Id_{\gr A}$.
\end{proposition}
\begin{proof}The construction is similar to that in the proof of Proposition~\ref{ciota}. We construct $\iota_{(\emptyset, 0)}$ and define $\iota_{(\emptyset, j)} = \SSS_A^{j} \iota_{(\emptyset, 0)} \SSS_A^{-j}$. Let $\RRR$ be the full subcategory of $\gr A$ whose objects are the canonical rank one projectives. We define $\iota_{(\emptyset, 0)}$ on $\RRR$, then use Lemmas~\ref{projfunc2} and \ref{projfunc} to extend to a functor defined on all of $\gr A$. The construction of $\iota_{(0,\emptyset)}$ is completely analogous.

Let $S = \{0, X_{\m }, Y_{\m }\}$ and let $\DDD$ be the full subcategory of $\gr A$ whose objects are the elements of $S$. Clearly $\DDD$ is closed under subobjects. Let $P$ be a graded rank one projective module. By Corollary~\ref{projXYZ}, $P$ surjects onto exactly one of $X_{\m }$ and $Y_{\m }$, a module that we called $F_0^{\m }(P)$. Hence, there exists a unique smallest submodule $N \subseteq P$ such that $P/N \in S$. Let $N = \iota_{(\emptyset, 0)} P$. By Lemma~\ref{projfunc3}, $\iota_{(\emptyset, 0)}$ gives an additive functor $\RRR \sra \gr A$ whose action on morphisms is given by restriction.

Focusing now on structure constants, let $\{c_i\}$ and $\{d_i\}$ be the structure constants for $P$ and $\iota_{(\emptyset, 0)} P$, respectively. We can compute $d_0$ by constructing the unique kernel to the surjection $P \sra F_0^{\m }(P)$. If $F_0^{\m }(P) = X_{\m }$, then
\[ (\iota_{(\emptyset, 0)} P)_i = \begin{cases} (z+ \m ) P_i & \mbox{ if } i \leq 0 \\
P_i & \mbox{ if } i >0
\end{cases}
\]
and if $F_0^{\m }(P) = Y$, then 
\[(\iota_{(\emptyset, 0)} P)_i = \begin{cases} P_i & \mbox{ if } i \leq 0 \\
(z+\m )P_i & \mbox{ if } i >0.
\end{cases}
\]
In particular, If $c_0 \in \{1, z \}$ then $d_0 = (z+\m ) c_0$ and if $c_0\in \{z+\m , f\}$, then $d_0 = (z+\m )^{-1} c_0$. By repeating this construction, we see that $\iota_{(\emptyset,0)}^2 P = (z+\m ) P$ and so $\gr A$ is an autoequivalence of $\gr A$ with quasi-inverse $\iota_{(\emptyset,0 )}$.

Because for any rank one graded projective module $P$ the structure constants for $\iota_{(\emptyset,0)} P$ differ from those of $P$ only in degree $0$, where they differ only by a factor of $z+\m $, by Lemma~\ref{dssc}, $\iota_{(\emptyset,0)} P$ has the same integrally supported simple factors as $P$ except if $F_0^{\m }(P) = X_\m $ then $F_0^{\m }(\iota_{(\emptyset,0)} P) = Y_\m $ and vice versa. Again, examining the action of $\iota_{(\emptyset,0)}$ over all rank one projectives $P$, we deduce that $\iota_{(\emptyset,0)} P$ has the claimed action on simple modules, fixing all but $X_{\m }$ and $Y_{\m }$.
\end{proof}

We have therefore also constructed analogues of Sierra's $\iota_j$ in the case that $\m \in \kk \setminus \ZZ$. The subscript on the $\iota$ keeps track of which of the simples modules is being permuted: the first coordinate corresponds to the shifts of $X_0$ and $Y_0$ while the second coordinate corresponds to the shifts of $X_{\m }$ and $Y_{\m }$. Observe that in this case the subgroup of $\Pic(\gr A)$ generated by the $\{\iota_{(j, \emptyset)}, \iota_{(\emptyset, j)} \}$ is isomorphic to the direct product $\ZZ_\fin \times \ZZ_\fin$. For every $J, J' \in \ZZ_\fin \times \ZZ_\fin$ we define
\[ \iota_{(J,J')} = \prod_{j \in J} \iota_{(j,\emptyset)} \prod_{j \in J'} \iota_{(\emptyset, j)}.
\]
For completeness, define $\iota_{(\emptyset, \emptyset)} = \Id_{\gr A}$.

Finally, we are able to determine $\Pic_0(\gr A)$ and therefore $\Pic(\gr A)$.

\begin{lemma} \label{pic0dihedral} Let $\m \in \NN$. Then the map
\begin{align*}\Phi: \ZZ_\fin &\sra \Pic_0\left(\gr A(f)\right) \\
J &\mapsto \iota_J
\end{align*}
is a group isomorphism.

Let $\m \in \kk \setminus \ZZ$. Then the map
\begin{align*}
\Psi: \ZZ_\fin \times \ZZ_\fin &\sra \Pic_0\left(\gr A(f)\right) \\
(J, J') &\mapsto \iota_{(J,J')}
\end{align*}
is a group isomorphism.
\end{lemma}
\begin{proof}First let $\m \in \NN$. For any $i \in \ZZ$, we saw that $\iota_i^2 \cong \Id_{\gr A}$. Hence, $\iota_J \iota_{J'} \cong \iota_{J \oplus J'}$ so $\Phi$ is a group homomorphism. It is clear that $\Phi$ is injective. To show surjectivity, suppose $\FFF \in \Pic_0(\gr A)$. Since $\FFF$ is numerically trivial, by Theorem~\ref{cautoequiv}, $\FFF$ fixes all shifts of $Z$.

By Lemma~\ref{anyproj}, we can construct a canonical rank one projective $P$ such that $F_n(P) = X\s{n}$ for all $n \geq 0$ and $F_n(P) = Y\s{n}$ for all $n < 0$. Note that since $\FFF$ is numerically trivial, for any $j \in \ZZ$, $\FFF(F_j(P)) \cong F_j(\FFF(P))$. Now by Corollary~\ref{largedegree}, $\FFF(F_j(P))$ can only differ from $F_j(\FFF(P))$ for finitely many $j$. Let $J$ be precisely those indices at which they differ. By Lemma~\ref{autosimple}, $\iota_J \cong \FFF$, so $\Phi$ is surjective.

The case $\m \in \kk \setminus \ZZ$ follows from the same argument, doubling the number of indices where necessary.
\end{proof}

\begin{theorem}\label{picthm} Let $f \in \kk[z]$ be quadratic. Then
\[ \Pic(\gr A(f)) \cong \ZZ_\fin \rtimes D_\infty.
\]
\end{theorem}
\begin{proof} This follows from Corollary~\ref{cpiccor}, Lemma~\ref{pic0dihedral}, and the fact that
\[ \ZZ_\fin \times \ZZ_\fin \cong \ZZ_\fin. \qedhere
\]
\end{proof}

Finally, we are interested in when the collections of modules $\{\iota_J A \mid J \in \ZZ_\fin\}$ (in the case $\m \in \NN$) or $\{ \iota_{(J,J')} \mid (J,J') \in \ZZ_\fin \times \ZZ_\fin \}$ (in the case $\m \in \kk \setminus \ZZ$) generate $\gr A(f)$. In the cases of a multiple root or non-congruent roots, then just as in the case for the first Weyl algebra, these collections of modules generate $\gr A(f)$. However, in the case of a congruent root, we see that $\{\iota_J A \mid J \in \ZZ_\fin\}$ does not.

\begin{lemma}\label{iotasgen} Let $f = z(z+\m )$.
\begin{enumerate}\item If $\m = 0$, then $\{\iota_J A \mid J \in \ZZ_\fin\}$ generates $\gr A(f)$.
\item If $\m \in \kk \setminus \ZZ$, then $\{ \iota_{(J, J')} A \mid (J, J') \in \ZZ_\fin \times \ZZ_\fin\}$ generates $\gr A(f)$.
\item If $\m \in \NNp$, then $\{\iota_J A \mid J \in \ZZ_\fin\}$ does not generate $\gr A(f)$.
\end{enumerate}
\end{lemma}
\begin{proof} In the first two cases, this follows from Proposition~\ref{projgen}, Proposition~\ref{ciota}, and Proposition~\ref{diota}. In the case that $\m \in \NNp$, for all $J \in \ZZ_\fin$, the shifts of $Z$ that are factors $\iota_J(A)$ are exactly those that are factors of $A$. Namely, $\Hom_{\gr A}(\iota_J(A), Z\s{n}) = 0$ for $n < 0$ and $n \geq \m $. Hence, no $\iota_J(A)$ has a surjection to $Z\s{-1}$, and so $\{\iota_J(A) \mid J \in \ZZ_\fin\}$ does not generate $\gr A(f)$.
\end{proof}

\begin{lemma} \label{cgrAorbits} Let $\m \in \NNp$. The action of $\Pic_0(\gr A)$ on the set of graded rank one projective modules has infinitely many orbits, one for each finite subset of $\ZZ$.
\end{lemma}
\begin{proof} By Lemma~\ref{anyproj}, if for each $n$ we choose $S_n \in \{X\s{n}, Y\s{n}, Z\s{n}\}$ such that for $n \gg 0$, $S_n = X\s{n}$ and $S_{-n} = Y\s{-n}$, then there exists a rank one projective whose integrally supported simple factors are precisely the $S_n$. By Theorem~\ref{cautoequiv}, an autoequivalence $\FFF \in \Pic_0(\gr A)$ will have $\FFF(Z \s{n}) \cong Z\s{n}$. Also, by Proposition~\ref{ciota}, there exist numerically trivial autoequivalences permuting $X\s{n}$ and $Y\s{n}$ for any $n$. Hence, for each finite subset $J$ of $\ZZ$ there is an orbit consisting of all rank one graded projective modules that surject onto exactly $Z\s{j}$ for each $j \in J$.
\end{proof}

\begin{remark}In the case $\m = 0$ or $\m \in \kk \setminus \ZZ$, it is easily checked that the action of $\Pic_0(\gr A)$ is transitive on the set of graded rank one projective modules.
\end{remark}

\bibliography{refs}{}
\bibliographystyle{amsalpha}

\end{document}